\documentclass[11pt,reqno]{amsart}

\usepackage{lineno,hyperref}

\modulolinenumbers[5]
\usepackage{amsmath}
\usepackage{amssymb}
\usepackage{mathrsfs}
\usepackage{amsthm}
\usepackage{bm}
\usepackage{graphicx}
\usepackage{booktabs}
\usepackage{float}
\usepackage{subfig}
\usepackage{geometry}
\usepackage{color}

\newtheorem{Def}{Definition}[section]
\newtheorem{lem}[Def]{Lemma}
\newtheorem{theo}[Def]{Theorem}

\newtheorem{rem}[Def]{Remark}

\newcommand{\ab}{a_{12}b_2-a_{22}b_1}
\newcommand{\tr}{\text{tr}}
\newcommand{\mcal}{\mathcal}
\newcommand{\DW}{\Delta W}
\newcommand{\mscr}{\mathscr}
\newcommand{\mbb}{\mathbb}
\newcommand{\mbf}{\mathbf}
\newcommand{\mrm}{\mathrm}
\newcommand{\ud}{\mathrm d}
\newcommand{\lam}{\Lambda(\lambda)}

\numberwithin{equation}{section}
\allowdisplaybreaks[4]

\begin{document}

\title[Superiority of stochastic symplectic methods via LDPs]{The probabilistic superiority of stochastic symplectic methods  via large deviations principles}

\author{Chuchu Chen}
\address{Academy of Mathematics and Systems Science, Chinese Academy of Sciences, Beijing
	100190, China; School of Mathematical Sciences, University of Chinese Academy of
	Sciences, Beijing 100049, China}
\email{chenchuchu@lsec.cc.ac.cn}
\author{Jialin Hong}
\address{Academy of Mathematics and Systems Science, Chinese Academy of Sciences, Beijing
	100190, China; School of Mathematical Sciences, University of Chinese Academy of
	Sciences, Beijing 100049, China}
\email{hjl@lsec.cc.ac.cn}

\author{Diancong Jin}
\address{Academy of Mathematics and Systems Science, Chinese Academy of Sciences, Beijing
	100190, China; School of Mathematical Sciences, University of Chinese Academy of
	Sciences, Beijing 100049, China}
\email{diancongjin@lsec.cc.ac.cn (Corresponding author)}

\author{Liying Sun}
\address{Academy of Mathematics and Systems Science, Chinese Academy of Sciences, Beijing
	100190, China; School of Mathematical Sciences, University of Chinese Academy of
	Sciences, Beijing 100049, China}
\email{liyingsun@lsec.cc.ac.cn}

\thanks{This work is supported by National Natural Science Foundation of China (NO. 91630312, NO.11711530071, NO.11871068 and NO.11971470).}

\keywords{
symplectic methods; superiority; large deviations principle;  rate function; asymptotical preservation.
}

\begin{abstract}	
It is well known that symplectic methods have been rigorously shown  to be superior to non-symplectic ones especially in long-time computation, when applied to deterministic Hamiltonian systems. 
In this paper, we attempt to study the probabilistic superiority of stochastic symplectic methods by means of the theory of large deviations. We propose the concept of asymptotical preservation of numerical methods for large deviations principles associated with the exact solutions of the general stochastic Hamiltonian systems. 
Concerning that  the linear stochastic oscillator is one  of the typical stochastic Hamiltonian systems,  
we take it as the test equation in this paper 
to obtain precise results about the rate functions of large deviations principles for both exact and numerical solutions.  
Based on the G\"artner--Ellis theorem, we first study the large deviations principles of the mean position and the mean velocity for both the exact solution  
and its numerical approximations. Then, we prove that stochastic symplectic methods asymptotically preserve these two large deviations principles, but non-symplectic ones do not. This indicates that stochastic symplectic methods are able to approximate well the exponential decay speed of the ``hitting probability" of the mean position and mean velocity of the  stochastic oscillator.  To the best of our knowledge,  this is the first result about using large deviations principle  to show the superiority of stochastic symplectic methods  compared with non-symplectic ones in the existing literature.

\end{abstract}

\maketitle

AMS subject classifications: 60F10, 60H35, 65C30, 65P10

\section{Introduction}  \label{Sec1}
A $2d$-dimensional  stochastic differential equation (SDE) is called a stochastic Hamiltonian system, if it can be written in the form: 
\begin{align}\label{SHS}
\ud\left(
\begin{array}{c}
p\\q
\end{array}
\right)=J^{-1}\nabla H_0(p,q)\ud t+\sum_{r=1}^{m}J^{-1}\nabla H_r(p,q)\circ \ud W_r(t), \quad J=\begin{bmatrix}
0&I_d\\
-I_d&0
\end{bmatrix},
\end{align}
where $\circ$ denotes the Stratonovich product, $H_i$, $i=0,1,\ldots,m$ are  smooth Hamilton functions, and $W=(W_1,\ldots,W_m)$ is an $m$-dimensional Brownian motion on a given complete filtered probability space $(\Omega,\mcal F,\{\mcal F_t\}_{t\geq0},\mbf P)$. The phase flow of \eqref{SHS} preserves the symplectic structure in phase space, i.e., $d p(t)\wedge d q(t)=d p(0)\wedge d q(0)$, a.s., for all $t\geq 0$. In order to preserve the symplectic structure, a class of numerical methods called stochastic symplectic methods are proposed ( \cite{MilsteinBook}).
In recent years, stochastic symplectic methods have received extensive attention, and large quantities of numerical experiments show that stochastic symplectic methods
possess excellent long-time stability  (see e.g., \cite{CHRK,CHJ,CHZ,CHLZ,Anton2,HongW,Wang2014,WHS}).  
One approach to theoretically explaining the superiority of stochastic symplectic methods is
based on  the techniques of modified equations and backward error analysis (see \cite{CohenVil,Anton19,Huang,Faou,HSW,Tony,WHS,Zyg} and references therein).
Different from their approach, we try to use  large deviations principle (LDP) to give a probabilistic interpretation for the superiority of the stochastic symplectic methods in this paper.

The theory of large deviations is concerned with the exponential decay of probabilities of very rare events, which can be regarded as  an extension or refinement of the law of large numbers and central limit theorem. It is usually used to describe the asymptotical behaviour of stochastic processes for which  large deviations estimates are concerned. 
If a stochastic process $\{X_T\}_{T>0}$ satisfies an LDP with the rate function $I$, then the \emph{hitting probability} $\mbf P\left(X_T\in[a,a+\ud a]\right)$ decays exponentially, i.e., $e^{-TI(a)}\ud a$. 
The rate function characterizes
the fluctuations of the stochastic process $\{X_T\}_{T>0}$ in the long-time limit, and has a wide range of applications in engineering and physical sciences (see e.g., \cite{Adaptive}).
When a numerical method is applied to a given stochastic differential equation,
it is worthwhile to study whether the numerical method can preserve asymptotically the decay rate $e^{-TI}$.

Let $\{Z_T\}_{T>0}$ be a stochastic process associated with the exact solution of \eqref{SHS}, usually viewed as an observable of \eqref{SHS}.   For a numerical method $\{p_n,q_n\}_{n\geq 0}$ for \eqref{SHS}, let $\{Z_N\}_{N\geq1}$ be a discrete approximation of $\{Z_T\}_{T>0}$ associated with the numerical method $\{p_n,q_n\}_{n\geq 0}$. For example, one can take $Z_T=\frac{1}{T}\int_{0}^{T}f(p(t),q(t))\ud t$ as an observable of \eqref{SHS} for some smooth function $f$. Then $Z_N=\frac{1}{N}\sum_{n=0}^{N-1}f(p_n,q_n)$ can be viewed as a discrete version of $Z_T$.  If $\{Z_T\}_{T>0}$ satisfies an LDP with the rate function $I$,  two natural questions arise: 
\begin{itemize}
	\item[(Q1)] Does $\{Z_N\}_{N\geq1}$ satisfy the  LDP with some rate function $I^h$ for fixed step-size? 
	
	\item[(Q2)] If so,  could  $I^h$ approximate $I$ well for sufficiently small $h$?
\end{itemize}	
Concerning the above questions,  we give the following definition on the asymptotical or even exact preservation of LDP.
\begin{Def}\label{Def1.1}
	Let $\{Z_T\}_{T>0}$ be a stochastic process associated with the exact solution of \eqref{SHS}. Let $\{Z_N\}_{N\geq1}$ be a discrete approximation of $\{Z_T\}_{T>0}$, associated with some numerical method $\{p_n,q_n\}_{n\geq 0}$ for \eqref{SHS}.  Assume that $\{Z_T\}_{T>0}$ and $\{Z_N\}_{N\geq1}$ satisfy the LDPs on some polish space $E$ with the rate function $I$ and $I^h$, respectively. We call $I^h_{mod}:=I^h/h$ the modified rate function of $I^h$.
	Moreover, method $\{p_n,q_n\}_{n\geq 0}$ is said to asymptotically preserve the LDP of $\{Z_T\}_{T>0}$, if
	\begin{equation}\label{asym}
	\lim_{h\to 0}I^h_{mod}(y)=I(y),\qquad\forall\quad y\in E,
	\end{equation}
	In particular, method $\{p_n,q_n\}_{n\geq 0}$ is said to exactly preserve the LDP of $\{Z_T\}_{T>0}$, if for all sufficiently small step-size $h$, $I^h_{mod}(\cdot)=I(\cdot)$.
\end{Def}

The definition  of the modified rate function $I^h_{mod}:=I^h/h$  in Definition \ref{Def1.1} is to unify the observation scale. In fact, if LDPs  for both $\{Z_T\}_{T>0}$  and $\{Z_N\}_{N\geq1}$ hold, then $\{Z_T\}_{T>0}$ and $\{Z_N\}_{N\geq1}$ formally satisfy
\begin{gather}
\mbf P\left(Z_T\in[a,a+\ud a]\right)\approx e^{-TI(a)}\ud a, ~ \text{for sufficiently large} ~ T,\label{PZT} \\
\mbf P(Z_N\in[a,a+\ud a])\approx e^{-NI^h(a)}\ud a=e^{-t_NI^h_{mod}(a)}\ud a, ~ \text{for sufficiently large}~ t_N. \label{PZN}
\end{gather} 
With $T=t_N$ being the observation scale, it is reasonable to use the modified rate function to evaluate the ability of the numerical method to preserve the large deviations rate functions.
Concerning that  the linear stochastic oscillator is one  of the typical stochastic Hamiltonian systems,  
we take it as the test equation in this paper 
to obtain precise results about the rate functions of LDPs for both exact and numerical solutions. 
Based on the G\"artner--Ellis theorem, we first study the LDPs of the mean position and the mean velocity for both the exact solution of the linear stochastic oscillator and its numerical approximations. Then, by giving the conditions which make numerical methods have at least first order convergence in mean-square sense, we prove that stochastic symplectic methods asymptotically preserve these two LDPs. 
We would like to mention that the conclusion is valid for general stochastic symplectic methods, not only for some specific ones. However, it is shown that neither of two LDPs is preserved asymptotically by non-symplectic methods based on the tail estimation of Gaussian random variables. To the best of our knowledge,  this is the first result about using LDP  to show the superiority of stochastic symplectic methods  compared with non-symplectic ones. 

The paper is organized as follows. In Section \ref{Sec2}, we give some basic concepts about LDP and establish the LDPs for both $\{A_T\}_{T>0}$ and $\{B_T\}_{T>0}$ of the linear stochastic oscillator. Sections \ref{Sec3} and \ref{Sec4} study the LDP for $\{A_N\}_{N\geq1}$ of general numerical methods, and show that symplectic methods asymptotically preserve the LDP for $\{A_T\}_{T>0}$. In Section \ref{Sec5}, by following the ideas of dealing with $\{A_N\}_{N\geq1}$, we investigate the LDP for $\{B_N\}_{N\geq1}$ and show that symplectic methods asymptotically preserve the LDP for $\{B_T\}_{T>0}$. In Section \ref{Sec6}, we verify our theoretical results by discussing about some concrete numerical methods, and construct some methods preserving exactly the LDPs for $\{A_T\}_{T>0}$ or $\{B_T\}_{T>0}$. These imply  the superiority of symplectic methods in preserving the LDPs for $\{A_T\}_{T>0}$ and $\{B_T\}_{T>0}$ of the linear stochastic oscillator. Finally, in Section \ref{Sec7}, we give our conclusions and propose several open problems for future study.

\section{LDPs for $\{A_T\}_{T>0}$ and $\{B_T\}_{T>0}$}\label{Sec2} 
In this section, we aim to prove that both the mean position $\{A_T\}_{T>0}$ and mean velocity $\{B_T\}_{T>0}$ of  the exact solution of our considered stochastic oscillator satisfy the LDPs. Before showing the LDPs of $\{A_T\}_{T>0}$ and $\{B_T\}_{T>0}$, we introduce some preliminaries upon the theory of large deviations, which can be found in \cite{Dembo,Achim}.

\begin{Def}\label{ratefun}
	$I:E\rightarrow[0,\infty]$ is called a rate function, if it is lower semicontinuous, where $E$ is a Polish space, i.e., complete and separable metric space. If all level sets $I^{-1}([-\infty,a])$, $a\in[0,\infty)$, are compact, then $I$ is called a good rate function.
\end{Def}
\begin{Def}\label{LDP}
	Let $I$ be a rate function and $(\mu_\epsilon)_{\epsilon>0}$ be a family of probability measures on $E$. We say that $(\mu_\epsilon)_{\epsilon>0}$ satisfies a large deviations principle (LDP) with rate function $I$ if
	\begin{flalign}
	(\rm{LDP 1})\qquad \qquad&\liminf_{\epsilon\to 0}\epsilon\log(\mu_\epsilon(U))\geq-\inf I(U)\qquad\text{for every open}~ U\subset E,\nonumber&\\
	(\rm{LDP 2})\qquad\qquad &\limsup_{\epsilon\to 0}\epsilon\log(\mu_\epsilon(C))\leq-\inf I(C)\qquad\text{for every closed}~ C\subset E.&\nonumber
	\end{flalign}
\end{Def}
Based on Definition \ref{LDP}, one can give the definition of LDP for a family of random variables similarly. Namely, let $\{X_\epsilon\}_{\epsilon>0}$ $\big(\text{resp.} ~\{X_n\}_{n\in\mbb N}\big)$ be a family of random variables from $\left(\Omega,\mathscr{F},\mbf P\right)$ to $(E,\mscr{B}(E))$. $\{X_\epsilon\}_{\epsilon>0}$ $\big(\text{resp.}~\{X_n\}_{n\in\mbb N}\big)$ is said to satisfy an LDP with the rate function $I$, if its distribution law $(\mbf P\circ X_\epsilon^{-1})_{\epsilon>0}$ $\big(\text{resp.}~(\mbf P\circ X_n^{-1})_{n\in\mbb N}\big)$ satisfies  (LDP$1$) and (LDP$2$) in Definition \ref{LDP}.  (see e.g., \cite{ChenX,Dembo})

The G\"artner--Ellis theorem plays an important role in dealing with the LDPs for a family of not independent random variables. When utilizing this theorem, one needs to examine whether the logarithmic moment generating function is essentially smooth.

\begin{Def} \label{Def1.4}
	A convex function  $\Lambda:\mbb R^d\to(-\infty,\infty]$ is essentially smooth if:
	
	$\mrm{(1)}$ $\mcal D_\Lambda^\circ$ is non-empty, where  $\mcal D_\Lambda^\circ$ is the interior of $\mcal D_\Lambda:=\{x\in\mbb R^d:\Lambda(x)<\infty\}$;
	
	$\mrm{(2)}$ $\Lambda(\cdot)$ is differentiable throughout $\mcal D_\Lambda^\circ$;
	
	$\mrm{(3)}$ $\Lambda(\cdot)$ is steep, namely, $\lim_{n\to\infty}\left|\nabla\Lambda(\lambda_n)\right|=\infty$ whenever $\{\lambda_n\}$ is a sequence in  $\mcal D_\Lambda^\circ$ converging to a boundary point of $\mcal D_\Lambda^\circ$. 
\end{Def}
\begin{theo}[G\"artner--Ellis]\label{G-E}
	Let $\{X_n\}_{n\in\mbb N}$ be a sequence of random vectors taking values in $\mbb R^d$. Assume that for each $\lambda\in \mbb R^d$, the logarithmic moment generating function, defined as the limit 
	$\Lambda(\lambda)\triangleq\lim_{n\to\infty}\frac{1}{n}\log\left(\mbf Ee^{n\langle\lambda,X_n\rangle}\right)$
	exists as an extended real number. Further, assume that the origin belongs to $\mcal D_\Lambda^\circ$.
	If  $\Lambda$ is an essentially smooth and lower semicontinuous function, then the LDP holds for $\{X_n\}_{n\in\mbb N}$ with the good rate function $\Lambda^*(\cdot)$. Here $\Lambda^*(x)=\underset{\lambda\in\mbb R^d}{\sup}\{\langle\lambda,x\rangle-\Lambda(\lambda)\}$, $x\in\mbb R^d$, is the Fenchel--Legendre transform of $\Lambda(\cdot)$.	
\end{theo}
It is known that the key point of the G\"artner--Ellis theorem is to study the  logarithmic moment generating function. Moreover, we would like to mention that the G\"artner--Ellis theorem is valid in the case of continuous parameter family  $\{X_\epsilon\}_{\epsilon>0}$ (see the remarks of \cite[Theorem 2.3.6]{Dembo}). 

The motivation of this paper is  to explain the superiority of stochastic symplectic methods, by studying the LDPs of numerical methods for a linear stochastic oscillator $
\ddot{X}_t+X_t=\alpha \dot{W}_t$
with $\alpha > 0,$ and $W_t$ being a $1$-dimensional standard Brownian motion defined on a complete filtered probability space $\left(\Omega,\mathscr{F},\{\mathscr{F}_t\}_{t\geq0},\mbf P\right)$. 
The linear stochastic oscillator can be rewritten as a $2$-dimensional stochastic Hamiltonian system
\begin{equation}\label{Eq1}
\ud \left(
\begin{array}{c}
X_t \\
Y_t
\end{array}
\right)=\left(
\begin{array}{cc}
0 & 1\\
-1& 0
\end{array}
\right)
\left(	\begin{array}{c}
X_t \\
Y_t
\end{array}\right)\ud t+\alpha
\left(\begin{array}{c}
0 \\
1
\end{array}\right)\ud W_t,\quad
\left(
\begin{array}{c}
X_0 \\
Y_0
\end{array}
\right)
=
\left(
\begin{array}{c}
x_0 \\
y_0
\end{array}
\right),
\end{equation}
whose phase flow preserves symplectic structure.
Namely, the oriented areas of the projection of the phase flow are invariant:
\begin{equation*}
\ud X_t\wedge\ud Y_t=\ud x_0\wedge\ud y_0, \qquad\forall\quad t\geq 0,
\end{equation*}
where the exact solution $\left(X_t,Y_t\right)$ of \eqref{Eq1} (see \cite[Chapter 8]{Mxr}) is
\begin{align}
X_t=&\phantom{-}x_0\cos(t)+y_0\sin(t)+\alpha \int_{0}^{t}\sin(t-s)\ud W_s, \label{Xt}\\
Y_t=&-x_0\sin(t)+y_0\cos(t)+\alpha \int_{0}^{t}\cos(t-s)\ud W_s.  \nonumber
\end{align}
To inherit the symplecticity of this stochastic oscillator, different kinds of symplectic methods have been constructed  (see \cite{Cohen,Tocino15} and references therein).

For SDE \eqref{Eq1}, we introduce the so-called mean position 
\begin{equation} \label{AT}
A_T=\frac{1}{T}\int_{0}^{T}X_t\,\ud t, \qquad \forall\quad T>0,
\end{equation} 
and the  mean velocity:
\begin{equation} \label{BT}
B_T=\frac{X_T}{T}, \qquad \forall\quad T>0.
\end{equation}
Both $A_T$ and $B_T$ are important observables, and they have many applications in physics. For example, 
the Ornstein--Uhlenbeck process is often used to describe the velocity of a particle moving in a random environment (\cite{OU}). In this case, $A_T$ can be interpreted as the mean value of the displacement process $\int_{0}^{T}X_t\,\ud t$, and $B_T$ as the mean value of velocity $X_t$ on the time interval $[0,T]$ (see also \cite{Adaptive}).
Next, by means of the G\"artner--Ellis theorem,  we show that  both the mean position $\{A_T\}_{T>0}$ and mean velocity $\{B_T\}_{T>0}$ of the exact solution satisfy the LDPs.
\begin{theo}\label{LDP for AT}
	$\{A_T\}_{T>0}$ satisfies an LDP with the good rate function 	$I(y)=\frac{y^2}{3\alpha^2}$, i.e., 
	\begin{eqnarray}
	\underset{T\to \infty}{\liminf}~\frac{1}{T}\log(\mbf P(A_T\in U))&\geq-\underset{y\in U}{\inf}I(y)\qquad&\text{for every open}~ U\subset \mbb R,\nonumber\\
	\underset{T\to \infty}{\limsup}~\frac{1}{T}\log(\mbf P(A_T\in C))&\leq-\underset{y\in C}{\inf}I(y)\qquad&\text{for every closed}~ C\subset \mbb R.\nonumber
	\end{eqnarray}
\end{theo}

\begin{proof}
	It follows from \eqref{Xt},  \eqref{AT} and the stochastic Fubini theorem that
	\begin{align}
	TA_T=\int_{0}^{T}X_t\ud t
	=x_0\sin(T)+y_0(1-\cos(T))+\alpha\int_{0}^{T}\left[1-\cos(T-s)\right]\ud W_s.
	\end{align}
	Thus, we have
	$\mathbf E\left[TA_T\right]=x_0\sin(T)+y_0(1-\cos(T)),$
	and
	\begin{equation*}
	\mbf{Var}\left[TA_T\right]=\alpha^2\int_{0}^{T}\left[1-\cos(T-s)\right]^2\ud s=\alpha^2\left[\frac{3T}{2}-2\sin(T)+\frac{\sin(2T)}{4}\right].
	\end{equation*} 
	Hence $\lambda TA_T\sim\mcal N\left(\lambda \mbf E\left[TA_T\right],\lambda^2\mbf{Var}\left[TA_T\right]\right)$ for every $\lambda\in\mbb R$. It follows from the characteristic function of $\lambda TA_T$ that
	$\mbf Ee^{\lambda TA_T}=e^{\lambda\mathbf E\left[TA_T\right]+\frac{\lambda^2}{2}\mbf{Var}\left[TA_T\right]}.$
	In this way, we obtain the logarithmic moment generating function
	$\lam=\lim_{T\to\infty}\frac{1}{T}\log\mbf Ee^{\lambda TA_T}
	=\frac{3\alpha^2}{4}\lambda^2$,
	which means that $\Lambda(\cdot)$ is an essentially smooth, lower semicontinuous function.
	Moreover, we have that the origin $0$ belongs to  $\mcal D_{\Lambda}^{\circ}=\mathbb R$.
	By the Theorem \ref{G-E}, we obtain that $\{A_T\}_{T>0}$ satisfies an LDP with the good rate function 
	$I(y)=\Lambda^*(y)=\underset{\lambda\in\mbb R}{\sup}\{y\lambda-\lam\}=\frac{y^2}{3\alpha^2}.$
\end{proof}

Notice that the LDP for $\{A_T\}_{T>0}$ is independent of the initial value $(x_0,y_0)$ of the stochastic oscillator \eqref{Eq1}. Theorem \ref{LDP for AT} indicates that, for any initial value $(x_0,y_0)$, the probability that the mean position $\{A_T\}_{T>0}$ hits the interval $[a,a+\ud a]$ decays exponentially and formally satisfies 
$\mbf P\left(A_T\in[a,a+\ud a]\right)\approx e^{-TI(a)}\ud a=e^{-T\frac{y^2}{3\alpha^2}}\ud a$, for sufficiently large $T$.

Similarly, we give the result of the LDP for $\{B_T\}_{T>0}$ in the following theorem.
\begin{theo}\label{LDP for BT}
	$\{B_T\}_{T>0}$ satisfies an LDP with the good rate function 	$J(y)=\frac{y^2}{\alpha^2}$.  
\end{theo}
\begin{proof}
	This proof is analogous to that of Theorem \ref{LDP for AT}, and hence is omitted.
\end{proof}
$\square$

The above two theorems give the LDPs of $\{A_T\}_{T>0}$ and $\{B_T\}_{T>0}$. 
For a numerical approximation $\{x_n,y_n\}$ of the linear stochastic oscillator \eqref{Eq1}, two natural questions are:  Do its discrete mean position $A_N=\frac{1}{N}\sum_{n=0}^{N-1}x_n$ and  discrete mean velocity $B_N=\frac{x_N}{Nh}$ satisfy similar LDPs as continuous case?
Is the method able to preserve or asymptotically preserve the LDPs of $\{A_T\}_{T>0}$ and $\{B_T\}_{T>0}$ in the sense that the modified rate functions converge to the rate functions of exact solution?
The next several sections of this paper are devoted  to answering the above questions.

\section{LDP for discrete mean position $\{A_N\}_{N\geq1}$} \label{Sec3}
In this section, we study the LDP for the discrete mean position of general numerical methods. We show that symplectic methods and non-symplectic ones satisfy different types of LDPs.

Let $\{(x_n,y_n)\}_{n\geq 1}$ be the discrete approximations at $t_n=nh$ with $x_n\approx X_{t_n}$,  $y_n\approx Y_{t_n}$, where $h>0$ is the given step-size. Following \cite{Tocino15}, we consider the general numerical methods in form of
\begin{equation}\label{Mthd}
\left(\begin{array}{c}
x_{n+1}\\y_{n+1}
\end{array}\right)
=A\left(\begin{array}{c}
x_n\\
y_n
\end{array}\right)+\alpha b\Delta W_n:
=\left(\begin{array}{cc}
a_{11}&a_{12}\\
a_{21}&a_{22}
\end{array}\right)
\left(\begin{array}{c}
x_n\\
y_n
\end{array}\right)
+\alpha\left(\begin{array}{c}
b_1\\
b_2
\end{array}\right)\Delta W_n,
\end{equation}
with $\Delta W_n=W_{t_{n+1}}-W_{t_n}$.  In fact, the real matrix $A$ and the real vector $b$  depend on both the method and  the constant step-size $h$.  In addition, we require $b_1^2+b_2^2\neq0$, which is natural since an effective numerical method for \eqref{Eq1} must depend on the Brownian motion. In the previous section, we derive the LDP for the mean position $\{A_T\}_{T>0}$ of the continuous system \eqref{Eq1}. In what follows, we consider the LDP for discrete mean position $\{A_N\}_{N\geq1}$ of the method \eqref{Mthd} and  study how closely the LDP for $\{A_N\}_{N\geq1}$ approximates the LDP for $\{A_T\}_{T>0}$. We recall that $A_N$ is defined as
\begin{equation}\label{AN}
A_N=\frac{1}{N}\sum\limits_{n=0}^{N-1}x_n,\qquad N=1,2,\dots
\end{equation}

Aiming at giving the general formula of $\{x_n\}$, we denote $M_n=\left(\begin{array}{c}
x_{n+1}\\x_n
\end{array}\right)$ for $n\geq 1$. It follows from recurrence \eqref{Mthd} that
$M_n=BM_{n-1}+r_n$, $n\geq 1,$ with 
\begin{equation*}
B=\left(\begin{array}{cc}
\text{tr}(A)&-\det(A)\\
1&0
\end{array}\right),\quad
r_n=\left(\begin{array}{c}
\alpha\left(b_1\Delta W_n+(a_{12}b_2-a_{22}b_1)\Delta W_{n-1}\right)\\
0
\end{array}\right),
\end{equation*}
where tr($A$) and det($A$) denote the trace and the determinant of $A$, respectively. In this way, we have
$M_n=B^nM_0+\sum_{j=1}^{n}B^{n-j}r_j$, $n\geq 1.$
Suppose that the coefficients in matrix $B$ satisfy
\begin{flalign*}
\mbf{(A1)} &\quad& \quad&4\det(A)-(\text{tr}(A))^2>0,&
\end{flalign*}
which guarantees that the eigenvalues of $B$ are
\begin{equation*}
\lambda_{\pm}=\frac{\text{tr}(A)}{2}\pm  \bm{i}\frac{\sqrt{4\det(A)-(\text{tr}(A))^2}}{2}=\sqrt{\det(A)}e^{\pm\bm{i}\theta},\qquad \bm{i}^2=-1,
\end{equation*}
for some $\theta\in(0,\pi)$ satisfying
\begin{equation}\label{theta}
\cos(\theta)=\frac{\text{tr}(A)}{2\sqrt{\det(A)}},\qquad\sin(\theta)=\frac{\sqrt{4\det(A)-(\text{tr}(A))^2}}{2\sqrt{\det(A)}}.
\end{equation}

Let $\hat{\alpha}_n=(\det(A))^{n/2}\frac{\sin((n+1)\theta)}{\sin(\theta)}$ and $\hat{\beta}_n=-(\det(A))^{(n+1)/2}\frac{\sin(n\theta)}{\sin(\theta)}$, for any integer $n$. It follows from the expression of  $M_n$ (one can refer to \cite{Tocino15}) that
\begin{equation*}
x_{n+1}=\hat{\alpha}_nx_1+\hat{\beta}_nx_0+\alpha\sum_{j=1}^{n}\hat{\alpha}_{n-j}\left[b_1\Delta W_j+(a_{12}b_2-a_{22}b_1)\Delta W_{j-1}\right],\qquad n\geq 0.
\end{equation*}
Since $x_1=a_{11}x_0+a_{12}y_0+\alpha b_1\Delta W_0$, $\hat{\alpha}_{-1}=0$ and $\hat{\alpha}_0=1$, for $n\geq 1$,
\begin{align}
x_n=&\left(a_{11}\hat{\alpha}_{n-1}+\hat{\beta}_{n-1}\right)x_0+a_{12}\hat{\alpha}_{n-1}y_0+\alpha b_1\hat{\alpha}_{n-1}\Delta W_0\nonumber\\
\phantom{=}&+\alpha\sum_{j=1}^{n-1}b_1\hat{\alpha}_{n-1-j}\Delta W_j+\alpha\sum_{j=0}^{n-2}(a_{12}b_2-a_{22}b_1)\hat{\alpha}_{n-2-j}\Delta W_j \nonumber\\
=&\left(a_{11}\hat{\alpha}_{n-1}+\hat{\beta}_{n-1}\right)x_0+a_{12}\hat{\alpha}_{n-1}y_0+\alpha\sum_{j=0}^{n-1}\left[b_1\hat{\alpha}_{n-1-j}+(\ab)\hat{\alpha}_{n-2-j}\right]\DW_j. \label{xn}
\end{align}

By \eqref{AN} and \eqref{xn}, we have
\begin{align}
NA_N=x_0+\sum_{n=1}^{N-1}x_n
=\left(1+a_{11}S^{\hat{\alpha}}_N+S^{\hat{\beta}}_N\right)x_0+a_{12}S^{\hat{\alpha}}_Ny_0+\alpha\sum_{j=0}^{N-2}c_j\DW_j,
\label{NAN}
\end{align}
where
$S^{\hat{\alpha}}_N=\sum_{n=0}^{N-2}\hat{\alpha}_n$, $S^{\hat{\beta}}_N=\sum_{n=0}^{N-2}\hat{\beta}_n$ and
\begin{align}
c_j:=&\sum_{n=j+1}^{N-1}\left[b_1\hat{\alpha}_{n-1-j}+(\ab)\hat{\alpha}_{n-2-j}\right]\nonumber\\
=&b_1\hat{\alpha}_{N-2-j}+(b_1+\ab)S^{\hat{\alpha}}_{N-1-j}.
\label{cj1}
\end{align}
To give  precise results of \eqref{NAN}, we need to compute  $S^{\hat{\alpha}}_N$ and $S^{\hat{\beta}}_N$ respectively, and then the following lemma is required (Its proof is given in the Appendix). 
\begin{lem}\label{lem1}
	For arbitrary $\theta\in(0,\pi)$, $N\in\mbb N^+$ and $a\in\mbb R$, it holds that
	\begin{equation}\label{sum1}
	\sum_{n=1}^{N}\sin(n\theta)a^n=\frac{a\sin(\theta)-a^{N+1}\sin((N+1)\theta)+a^{N+2}\sin(N\theta)}{1-2a\cos(\theta)+a^2}.
	\end{equation}
	In particular, if $a=1$, then
	\begin{equation}\label{sum2}
	\sum_{n=1}^{N}\sin(n\theta)=\frac{\cos\left(\frac{\theta}{2}\right)-\cos((N+\frac{1}{2})\theta)}{2\sin\left(\frac{\theta}{2}\right)}.
	\end{equation}
\end{lem}

It follows from \eqref{sum1} that, for every $N\geq 1$,
{\small
	\begin{align}
	S^{\hat{\alpha}}_N
	=\frac{\sin(\theta)-\left(\sqrt{\det(A)}\right)^{N-1}\sin(N\theta)+\left(\sqrt{\det(A)}\right)^{N}\sin((N-1)\theta)}{\sin(\theta)\left(1-2\sqrt{\det(A)}\cos(\theta)+\det(A)\right)}. \label{Salpha}
	\end{align}
}
Further, because $\hat{\beta}_n=-\det(A)\hat{\alpha}_{n-1}$ and $\alpha_{-1}=0$, we have
{\small
	\begin{align}
	S^{\hat{\beta}}_N
	=-\frac{\det(A)\sin(\theta)-\left(\sqrt{\det(A)}\right)^{N}\sin((N-1)\theta)+\left(\sqrt{\det(A)}\right)^{N+1}\sin((N-2)\theta)}{\sin(\theta)\left(1-2\sqrt{\det(A)}\cos(\theta)+\det(A)\right)},\nonumber\\
	\label{Sbeta}
	\end{align}
}
and
{\small
	\begin{align}
	c_j=&\frac{b_1}{\sin(\theta)}\sin((N-1-j)\theta)\left(\sqrt{\det(A)}\right)^{N-2-j}+\frac{b_1+\ab}{\sin(
		\theta)}\cdot \nonumber \\
	\phantom{=}&\frac{\sin(\theta)-\left(\sqrt{\det(A)}\right)^{N-2-j}\sin((N-1-j)\theta)+\left(\sqrt{\det(A)}\right)^{N-1-j}\sin((N-2-j))\theta)}{1-2\sqrt{\det(A)}\cos(\theta)+\det(A)}. \nonumber \\ \label{cj2}
	\end{align}
}

Based on \eqref{Salpha}-\eqref{cj2}, we obtain the expression of $NA_N$. Next, we study the LDP of $\{A_N\}_{N\geq1}$ for symplectic methods and non-symplectic ones, respectively. It is known that the method \eqref{Mthd} preserves the symplectic structure, i.e., $\ud x_{n+1}\wedge\ud y_{n+1}=\ud x_{n}\wedge\ud y_{n}$, if and only if $\det(A)=1$ (In fact, this condition is equivalent to that method \eqref{Mthd} preserves the phase volume).
In addition, for non-symplectic methods,
we exclude the case $\det(A)>1$ which makes $S^{\hat{\alpha}}_N$, $S^{\hat{\beta}}_N$ and $c_N$ exponentially increase as $N$ increases. This is to say, we need to deal with the case $\det(A)=1$ and the case $\det(A)<1$ separately.  
\subsection{LDP of $\{A_N\}_{N\geq1}$ for symplectic methods}
In this part, we derive the LDP for $\{A_N\}_{N\geq1}$ of the method \eqref{Mthd} in the case of preserving the symplecticity. Hereafter we use the notation $K(a_1,\ldots,a_m)$ to denote some constant dependent on the parameters $a_1,\ldots,a_m$ but independent of $N$, which may vary from one line to another.

We assume that 
\begin{flalign*}
\mbf{(A2)} &\quad&\quad&\det(A)=1.&
\end{flalign*}
Under $\mbf{(A2)}$, we have $\hat{\alpha}_n=\frac{\sin((n+1)\theta)}{\sin(\theta)}$, $\hat{\beta}_n=-\frac{\sin(n\theta)}{\sin(\theta)}$. Then by \eqref{cj1} and \eqref{sum2}, we obtain
{\small
	\begin{gather}
	S^{\hat{\alpha}}_N=\frac{\cos\left(\frac{\theta}{2}\right)-\cos((N-\frac{1}{2})\theta)}{2\sin(\theta)\sin\left(\frac{\theta}{2}\right)},\quad 
	S^{\hat{\beta}}_N=-\frac{\cos\left(\frac{\theta}{2}\right)-\cos((N-\frac{3}{2})\theta)}{2\sin(\theta)\sin\left(\frac{\theta}{2}\right)}, \label{SUM1} \\
	c_j=\frac{(b_1+\ab)\cos\left(\frac{\theta}{2}\right)-b_1\cos((N-\frac{1}{2}-j)\theta)-(\ab)\cos((N-\frac{3}{2}-j)\theta)}{2\sin(\theta)\sin\left(\frac{\theta}{2}\right)}. \label{SUM3}
	\end{gather}
}
By \eqref{SUM1}, it holds that
$|S^{\hat{\alpha}}_N|+|S^{\hat{\beta}}_N|\leq K(\theta),$ for each $N\geq2.$
According to the increment independence of Brownian motions, it follows from \eqref{NAN} that $NA_N$ is Gaussian.
Further, it follows from \eqref{NAN} and \eqref{SUM3} that
\begin{gather} \label{ENAN}
\left|\mbf E[NA_N]\right|=\left|\left(1+a_{11}S^{\hat{\alpha}}_N+S^{\hat{\beta}}_N\right)x_0+a_{12}S^{\hat{\alpha}}_Ny_0\right|\leq K(x_0,y_0,\theta),\\
\label{Var1}
\mbf{Var}[NA_N]=\alpha^2h\sum_{j=0}^{N-2}c_j^2=\frac{\alpha^2h}{4\sin^2(\theta)\sin^2\left(\frac{\theta}{2}\right)}\sum_{j=0}^{N-2}\tilde{c}_j^2,
\end{gather}
with
{\small
	\begin{align}
	\tilde{c}^2_j=
	(b_1+\ab)^2\cos^2\left(\frac{\theta}{2}\right)+\frac{1}{2}b_1^2+\frac{1}{2}(\ab)^2+b_1(\ab)\cos(\theta)+R_j, \nonumber\\\label{CJ1}
	\end{align}
}
where
\begin{align}
R_j=&\frac{b_1^2}{2}\cos((2N-1-2j)\theta)+\frac{(\ab)^2}{2}\cos((2N-3-2j)\theta) \nonumber \\
\phantom{=}&-2b_1(b_1+\ab)\cos\left(\frac{\theta}{2}\right)\cos\left(\frac{(2N-1-2j)\theta}{2}\right) \nonumber\\
\phantom{=}&-2(b_1+\ab)(\ab)\cos\left(\frac{\theta}{2}\right)\cos\left(\frac{(2N-3-2j)\theta}{2}\right) \nonumber \\
\phantom{=}&+b_1(\ab)\cos((2N-2-2j)\theta). \nonumber
\end{align}
We claim
$\left|\sum_{j=0}^{N-2}R_j\right|\leq K(\theta).$
In detail, by  $\sum_{n=1}^{N}\cos\left((2n+1)\theta\right)=\frac{\sin\left((2N+2)\theta\right)-\sin(2\theta)}{2\sin(\theta)}$, we have
\begin{equation*}
\left|\sum_{j=0}^{N-2}\cos((2N-1-2j)\theta)\right|=\left|\sum_{n=1}^{N-1}\cos((2n+1)\theta)\right|=\left|\frac{\sin(2N\theta)-\sin(2\theta)}{2\sin(\theta)}\right|\leq K(\theta).
\end{equation*}
Analogously, we obtain
$
\left|\sum_{j=0}^{N-2}\cos((2N-3-2j)\theta)\right|+\left|\sum_{j=0}^{N-2}\cos\left(\frac{(2N-1-2j)\theta}{2}\right)\right|  
+\left|\sum_{j=0}^{N-2}\cos\left(\frac{(2N-3-2j)\theta}{2}\right)\right|+\left|\sum_{j=0}^{N-2}\cos((2N-2-2j)\theta)\right|\leq K(\theta), \nonumber
$
which proves the above claim.

Based on \eqref{ENAN}, \eqref{Var1}, \eqref{CJ1} and $\left|\sum_{j=0}^{N-2}R_j\right|\leq K(\theta)$,  we have
\begin{align}
\Lambda^h(\lambda):=&\lim_{N\to\infty}\frac{1}{N}\log\mbf Ee^{\lambda NA_N}\nonumber\\
=&\frac{\alpha^2h\lambda^2}{8\sin^2(\theta)\sin^2\left(\frac{\theta}{2}\right)}\left[(b_1+\ab)^2\cos^2\left(\frac{\theta}{2}\right)+\frac{1}{2}b_1^2\right.\nonumber\\
\phantom{=}&\left.+\frac{1}{2}(\ab)^2+b_1(\ab)\cos(\theta)\right]. \label{lamdah1}
\end{align}
As a result of \eqref{theta} with $\det(A)=1$, it holds that
\begin{gather}
\cos(\theta)=\frac{\tr(A)}{2},\quad\sin(\theta)=\frac{\sqrt{4-(\tr(A))^2}}{2}, \nonumber\\
\sin^2\left(\frac{\theta}{2}\right)=\frac{1-\cos(\theta)}{2}=\frac{2-\tr(A)}{4},\quad\cos^2\left(\frac{\theta}{2}\right)=\frac{1+\cos(\theta)}{2}=\frac{2+\tr(A)}{4}.\label{cossin}
\end{gather}
Substituting \eqref{cossin} into \eqref{lamdah1} yields that
\begin{align}
\Lambda^h(\lambda)=&
\frac{\alpha^2h\lambda^2}{2(2+\tr(A))(2-\tr(A))^2}\left[(b_1+\ab)^2(4+\tr(A))\right.\nonumber\\
\phantom{=}&\left.-2b_1(\ab)(2-\tr(A))\right]. \label{symlog}
\end{align}
In order to show that $\Lambda^h$ is essentially smooth. We need to use the following lemma, whose proof is given in the Appendix.
\begin{lem}\label{lem3.2}
	Under assumptions $\mbf{(A1)}$ and $\mbf{(A2)}$, we have
	
	$(1)$ $b_1^2+(\ab)^2\neq0$;
	
	$(2)$ $(b_1+\ab)^2(4+\rm{tr}(A))-2b_1(\ab)(2-\rm{tr}(A))>0.$
\end{lem}
Lemma \ref{lem3.2}(2) means that $\Lambda^h$ is essentially smooth. It follows from Theorem \ref{G-E} that $\{A_N\}_{N\geq1}$ satisfies an LDP with the good rate function
\begin{align}
I^h(y)=&\underset{\lambda\in\mbb R}{\sup}\{y\lambda-\Lambda^h(\lambda)\}  \nonumber\\
=&\frac{(2+\tr(A))(2-\tr(A))^2y^2}{2\alpha^2h\left[(b_1+\ab)^2(4+\tr(A))-2b_1(\ab)(2-\tr(A))\right]}.\label{LDP1I}
\end{align}
Finally, we acquire the following theorem:
\begin{theo}\label{LDP1}
	If the numerical method \eqref{Mthd} for approximating the SDE \eqref{Eq1} satisfies the assumptions $\mbf{(A1)}$ and $\mbf{(A2)}$,  then its mean position $\{A_N\}_{N\geq1}$ satisfies an LDP with the good rate function
	given by \eqref{LDP1I}.
\end{theo}
\begin{rem}\label{rem3.3}
	Theorem \ref{LDP1} indicates that to make the LDP hold for $\{A_N\}_{N\geq1}$, the step-size $h$ need to be restricted such that conditions $\mbf{(A1)}$ and $\mbf{(A2)}$  hold. Moreover, the rate function $I^h(y)$ does not depend on the initial $(x_0,y_0)$. That is to say, for appropriate step-size $h$ and arbitrary initial value, $\{A_N\}_{N\geq1}$ formally satisfies $\mbf P(A_N\in[a,a+\ud a])\approx e^{-NI^h(a)}\ud a$ for sufficiently large $N$.
\end{rem}
\subsection{LDP of $\{A_N\}_{N\geq1}$ for non-symplectic methods}
In this part, we show the LDP for $\{A_N\}_{N\geq1}$ of method \eqref{Mthd} when it does not preserve the symplecticity. To this end, we firstly suppose that
\begin{flalign*}
\mbf{(A3)} &\quad&\quad&0<\det(A)<1.&
\end{flalign*}
Under condition $\mbf{(A3)}$, one immediately concludes from \eqref{Salpha} and \eqref{Sbeta} that
$\left|S^{\hat{\alpha}}_N\right|+\left|S^{\hat{\beta}}_N\right|\leq K(\theta)$, for all $N\geq 2$,
which gives
\begin{equation} \label{3.2E}
\left|\mbf E[NA_N]\right|\leq K(x_0,y_0,\theta).
\end{equation}
It follows from \eqref{NAN} and \eqref{cj2} that
\begin{equation} \label{3.2Var}
\mbf{Var}(NA_N)=\alpha^2h\sum_{j=0}^{N-2}c_j^2,
\end{equation}
where
\begin{equation} \label{3.2cj^2} c_j^2=\left(\frac{b_1+\ab}{1-2\sqrt{\det(A)}\cos(\theta)+\det(A)}\right)^2+\tilde{R}_j,
\end{equation}
with
{\scriptsize
	\begin{align}
	\tilde{R}_j 
	=&\frac{b_1^2\sin^2((N-1-j)\theta)(\det(A))^{N-2-j}}{\sin^2(\theta)}+\frac{(b_1+\ab)^2}{\sin^2(\theta)\left(1-2\sqrt{\det(A)}\cos(\theta)+\det(A)\right)^2}\cdot\nonumber \\
	\phantom{=}&\left[(\det(A))^{N-2-j}\sin^2((N-1-j)\theta)+(\det(A))^{N-1-j}\sin^2((N-2-j)\theta)-2\sin(\theta)\left(\sqrt{\det(A)}\right)^{N-2-j}\sin((N-1-j)\theta)\right. \nonumber\\
	\phantom{=}&\left.+2\sin(\theta)\left(\sqrt{\det(A)}\right)^{N-1-j}\sin((N+2-j)\theta)-2\left(\sqrt{\det(A)}\right)^{2N-3-2j}\sin((N-1-j)\theta)\sin((N-2-j)\theta)\right] \nonumber\\
	\phantom{=}&+\frac{2b_1(b_1+\ab)}{\sin^2(\theta)\left(1-2\sqrt{\det(A)}\cos(\theta)+\det(A)\right)}\left[\sqrt{\det(A)}^{N-2-j}\sin(\theta)\sin((N-1-j)\theta)\right. \nonumber\\
	\phantom{=}&\left.-(\det(A))^{N-2-j}\sin^2((N-1-j)\theta))+\left(\sqrt{\det(A)}\right)^{2N-3-2j}\sin((N-1-j)\theta)\sin((N-2-j)\theta
	)\right].\nonumber
	\end{align}
}
Moreover, it holds that
\begin{equation} \label{3.2SUMR}
\left|\sum_{j=0}^{N-2}\tilde{R}_j\right|\leq K(\theta)\sum_{j=0}^{N}\left(\sqrt{\det(A)}\right)^j\leq K(\theta).
\end{equation}
Combining \eqref{3.2E}, \eqref{3.2Var}, \eqref{3.2cj^2} and \eqref{3.2SUMR} leads to
\begin{align}
\widetilde{\Lambda}^h(\lambda)=&\lim_{N\to\infty}\frac{1}{N}\log\mbf Ee^{\lambda NA_N}  \nonumber\\
=&\frac{\alpha^2h\lambda^2}{2}\lim_{N\to\infty}\frac{1}{N}\left[\left(\frac{b_1+\ab}{1-2\sqrt{\det(A)}\cos(\theta)+\det(A)}\right)^2(N-1)+\sum_{j=0}^{N-2}\tilde{R}_j\right]\nonumber\\
=&\frac{\alpha^2h\lambda^2}{2}\left(\frac{b_1+\ab}{1-2\sqrt{\det(A)}\cos(\theta)+\det(A)}\right)^2\nonumber.
\end{align}
If we assume that
\begin{flalign*}
\mbf{(A4)} &\quad&\quad&b_1+\ab\neq0,&
\end{flalign*}
then it follows from Theorem \ref{G-E} that $\{A_N\}_{N\geq 1}$ satisfies an LDP with the good rate function
$\widetilde{I}^h(y)=\frac{y^2}{2\alpha^2h}\left(\frac{1-2\sqrt{\det(A)}\cos(\theta)+\det(A)}{b_1+\ab}\right)^2 
=\frac{y^2}{2\alpha^2h}\left(\frac{1-\tr(A)+\det(A)}{b_1+\ab}\right)^2,  $
where we have used \eqref{theta} in the second equality.
Finally, we obtain the following theorem:
\begin{theo} \label{LDP2}
	If the numerical method \eqref{Mthd} for approximating the SDE \eqref{Eq1} satisfies the assumptions $\mbf{(A1)}$, $\mbf{(A3)}$ and $\mbf{(A4)}$, then its mean position $\{A_N\}_{N\geq1}$ satisfies an LDP with the good rate function
	$\widetilde I^h(y)=\frac{y^2}{2\alpha^2h}\left(\frac{1-\mrm{tr}(A)+\det(A)}{b_1+\ab}\right)^2.$
\end{theo}

\vspace{5mm}
\section{Asymptotical preservation for the LDP of $\{A_T\}_{T>0}$} \label{Sec4}
In Section \ref{Sec3}, we acquire the LDP for mean position $\{A_N\}_{N\geq1}$
when the method \eqref{Mthd} is symplectic or non-symplectic separately, for given appropriate step-size. In this section, we study their asymptotical preservation for the LDP of $\{A_T\}_{T>0}$ as step-size tends to $0$ (see Definition \ref{Def1.1}).  
By Definition \ref{Def1.1}, we obtain the modified rate functions of the rate functions appearing in Theorems \ref{LDP1} and \ref{LDP2}, respectively, as follows:
\begin{gather}
I_{mod}^h(y)=\frac{(2+\mrm{tr}(A))(2-\mrm{tr}(A))^2y^2}{2\alpha^2h^2\left[(b_1+\ab)^2(4+\mrm{tr}(A))-2b_1(\ab)(2-\mrm{tr}(A))\right]},   \label{mod1}\\
\widetilde I_{mod}^h(y)=\frac{y^2}{2\alpha^2h^2}\left(\frac{1-\mrm{tr}(A)+\det(A)}{b_1+\ab}\right)^2.  \label{mod2}
\end{gather}

It would fail to get the asymptotically convergence for $I_{mod}^h(y)$ and $I_{mod}^h(y)$ only by means of conditions $\mbf{(A1)}-\mbf{(A4)}$ in two aspects:
one is that both $A$ and $b$ are some functions of step-size $h$, which are unknown unless a specific method is applied;
the other is that for some $A$ and $b$, the numerical approximation may not be convergent to the original system.
A solution to this problem is studying the convergence on finite interval of numerical methods. In what follows, we consider the mean-sqaure convergence of the method \eqref{Mthd}. 

For the sake of simplicity, we first give some notations. Let $R=\mcal{O}(h^p)$ stand for $\left|R\right|\leq Ch^p$, for all sufficiently small step-size $h$, where $C$ is independent of $h$ and may vary from one line to another. $f(h)\sim h^p$ means that $f(h)$ and $h^p$ are equivalent infinitesimal. Furthermore, $\left\|\cdot\right\|_2$ denotes $2$-norm of a vector or matrix and $\left\|\cdot\right\|_F$ denotes Frobenius norm of a matrix. 

Since \eqref{Eq1} is driven by the additive  noise, the mean-square convergence order of general numerical methods which are known for the moment to approximate this system is no less than $1$. Hence, in what follows, we restrict \eqref{Mthd} to the numerical method with at least first order convergence in mean-square sense.
To give the conditions about the mean-square convergence of the method \eqref{Mthd}, we introduce the Euler-Maruyama method of form \eqref{Mthd} with
$
A^{EM}=\left(\begin{array}{cc}
1&h\\
-h&1
\end{array}\right)$, $
b^{EM}=\left(\begin{array}{c}
0\\1
\end{array}\right).
$
Based on the fundamental convergence theorem, we acquire the sufficient conditions which make numerical method \eqref{Mthd} have at least first order convergence in mean-square sense. 
\begin{theo} \label{tho4.1}
	If the numerical method \eqref{Mthd} satisfies
	\begin{equation} \label{Con}
	\left\|A-A^{EM}\right\|_F=\mcal{O}(h^2)\qquad \text{and}\qquad \left\|b-b^{EM}\right\|_2=\mcal{O}(h),
	\end{equation}
	then its  convergence order is at least $1$ in mean-square sense on any finite interval $[0,T_0]$, i.e.,
	$\underset{n\geq0,~nh\leq T_0}{\sup} \left[\mbf E\left(\left(x_n-X(t_n)\right)^2+\left(y_n-Y(t_n)\right)^2\right)\right]^{1/2}\leq K(T_0)h.$	
\end{theo}

We put the proof of this theorem into the Appendix. By the definitions of $2$-norm and Frobenius norm, \eqref{Con} is equivalent to
\begin{flalign*}
\mbf{(B)} &~&~&\left|a_{11}-1\right|+\left|a_{22}-1\right|+\left|a_{12}-h\right|+\left|a_{21}+h\right|=\mcal{O}(h^2),~\text{and}~\left|b_1\right|+\left|b_2-1\right|=\mcal O(h).&
\end{flalign*}
Using this condition $\mbf{(B)}$, we have the following lemma (its proof is given in the Appendix), which is used to study whether method \eqref{Mthd} asymptotically preserves the LDPs for $\{A_T\}_{T>0}$ or $\{B_T\}_{T>0}$ of exact solution.
\begin{lem} \label{lem4.3}
	Under the condition $\mbf{(B)}$,  the following properties hold:
	
	$\mrm{(1)}$ $\mrm{tr}(A)\to 2$ as $h\to 0$;	
	
	$\mrm{(2)}$ $\left(1-\mrm{tr}(A)+\det(A)\right)\sim h^2$;
	
	$\mrm{(3)}$ $\left(b_1+\ab\right)\sim h$.
\end{lem}

By Lemma \ref{lem4.3}, we  obtain the convergence of the modified rate functions in \eqref{mod1} and \eqref{mod2}. 

Case $1$: Let $\mbf{(A1)}$, $\mbf{(A2)}$ and $\mbf{(B)}$ hold. Noting $\det(A)=1$ in this case, Lemma \ref{lem4.3}(2) yields $\left(2-\tr(A)\right)\sim h^2$. Hence,
\begin{equation} \label{lim}
\lim_{h\to 0}\frac{b_1(\ab)\left(2-\tr(A)\right)}{h^2}=0.
\end{equation}
It follows from Lemma \ref{lem4.3}, \eqref{mod1} and \eqref{lim} that
\begin{align}
\phantom{=}&\lim_{h\to 0}I^h_{mod}(y) \nonumber\\
=&\frac{y^2}{2\alpha^2}\frac{\lim_{h\to 0}\left(2+\tr(A)\right)}{\lim_{h\to 0}(4+\mrm{tr}(A))(b_1+\ab)^2/h^2-2\lim_{h\to 0}b_1(\ab)(2-\mrm{tr}(A))/h^2} \nonumber\\
=&\frac{y^2}{3\alpha^2}. \label{mod1lim}
\end{align}

Case $2$: Let $\mbf{(A1)}$, $\mbf{(A3)}$, $\mbf{(A4)}$ and $\mbf{(B)}$ hold. According to \eqref{mod2} and Lemma \ref{lem4.3}, we have
$\lim_{h\to 0}\widetilde{I}^h_{mod}=\frac{y^2}{2\alpha^2}\lim_{h\to 0}\frac{(h^2)^2}{h^2\cdot h^2}=\frac{y^2}{2\alpha^2}.$
Therefore, by Definition \ref{Def1.1}, we get the following two theorems.
\begin{theo} \label{MLDP1}
	For the numerical method \eqref{Mthd} approximating the stochastic oscillator \eqref{Eq1}, if the assumptions $\mbf{(A1)}$ and $\mbf{(A2)}$  hold, then we have 
	
	$\mrm{(1)}$ The method \eqref{Mthd} is symplectic;
	
	$\mrm{(2)}$ The discrete mean position $\{A_N\}_{N\geq1}$ of  method \eqref{Mthd} satisfies an LDP with the good rate function
	\begin{equation} \label{Ih1}
	I^h(y)=\frac{(2+\mrm{tr}(A))(2-\mrm{tr}(A))^2y^2}{2\alpha^2h\left[(b_1+\ab)^2(4+\mrm{tr}(A))-2b_1(\ab)(2-\mrm{tr}(A))\right]};
	\end{equation}
	
	$\mrm{(3)}$ Moreover, if assumption $\mbf{(B)}$ holds, then method \eqref{Mthd} asymptotically preserves the LDP of $\{A_T\}_{T>0}$, i.e., the modified rate function $I^h_{mod}(y)=I^h(y)/h$ satisfies:
	\begin{equation*}
	\lim_{h\to 0}I^h_{mod}(y)=I(y),\qquad\forall\quad y\in\mbb R,
	\end{equation*}
	where $I(\cdot)$ is the rate function of LDP for $\{A_T\}_{T>0}$.
\end{theo}

\begin{theo} \label{MLDP2}
	For the numerical method \eqref{Mthd} approximating the stochastic oscillator \eqref{Eq1}, if the assumptions $\mbf{(A1)}$, $\mbf{(A3)}$ and $\mbf{(A4)}$ hold, then we have 
	
	$\mrm{(1)}$ The method \eqref{Mthd} is non-symplectic;
	
	$\mrm{(2)}$ The discrete mean position $\{A_N\}_{N\geq1}$ of  method \eqref{Mthd} satisfies an LDP with the good rate function
	$	\widetilde I^h(y)=\frac{y^2}{2\alpha^2h}\left(\frac{1-\mrm{tr}(A)+\det(A)}{b_1+\ab}\right)^2;$
	
	$\mrm{(3)}$ Moreover, if assumption $\mbf{(B)}$ holds, then method \eqref{Mthd} does not asymptotically preserve the LDP of $\{A_T\}_{T>0}$, i.e., for $y\neq0$,
	$	\lim_{h\to 0}\widetilde I^h_{mod}(y)\neq I(y),$
	where $\widetilde I^h_{mod}(y)=\widetilde I^h(y)/h$, and $I(\cdot)$ is the rate function of LDP for $\{A_T\}_{T>0}$.
\end{theo}
\begin{rem}
	Theorems \ref{MLDP1} and \ref{MLDP2} indicate that under appropriate conditions, the  symplectic methods asymptotically preserve the LDP for the mean position $\{A_T\}_{T>0}$ of original system \eqref{Eq1}, while the  non-symplectic methods do not. This implies that, in comparison with non-symplectic methods, symplectic methods have  long-time stability in the aspect of LDP for the mean position.
\end{rem}
\vspace{5mm}

\section{LDP for discrete mean velocity $\{B_N\}_{N\geq1}$} \label{Sec5}
In Section \ref{Sec2}, we obtain the LDP for mean velocity $\{B_T\}_{T>0}$ of original system \eqref{Eq1}. In this section, following the ideas of dealing with discrete mean position, we investigate the LDP for discrete mean velocity. 

We consider the numerical approximation of $B_T=\frac{X_T}{T}$ at $t_N=Nh$. Noting that $x_N$ is used to approximate $X_{t_N}$ in terms of numerical method \eqref{Mthd}, we define discrete mean velocity as
\begin{equation}\label{BN}
B_N=\frac{x_N}{Nh},\qquad N=1,2,\dots.
\end{equation}
In what follows, we study the LDP for $\{B_N\}_{N\geq1}$ of method \eqref{Mthd} and its asymptotical preservation for LDP of $\{B_T\}_{T>0}$. Similar to the arguments on $\{A_N\}_{N\geq1}$, we  introduce the modified rate function to characterize how the LDP for $\{B_N\}_{N\geq1}$ approximates the LDP for $\{B_T\}_{T>0}$.

We still assume that $\mbf{(A1)}$ holds. In this case, the equality \eqref{xn} holds. Then
\begin{equation} \label{xN}
x_N=\left(a_{11}\hat{\alpha}_{N-1}+\hat{\beta}_{N-1}\right)x_0+a_{12}\hat{\alpha}_{N-1}y_0+\alpha\sum_{n=0}^{N-1}\left[b_1\hat{\alpha}_{N-1-n}+(\ab)\hat{\alpha}_{N-2-n}\right]\DW_n
\end{equation} 
with
\begin{gather*}
\hat{\alpha}_n=\left(\det(A)\right)^{n/2}\frac{\sin((n+1)\theta)}{\sin(\theta)},\qquad \hat{\beta}_n=-\left(\det(A)\right)^{\frac{n+1}{2}}\frac{\sin(n\theta)}{\sin(\theta)}.
\end{gather*}
According to \eqref{xN}, $x_N$ is Gaussian whose expectation is
{\small
	\begin{align}
	\mbf E(x_N)=
	\left(a_{11}\left(\det(A)\right)^{\frac{N-1}{2}}\frac{\sin(N\theta)}{\sin(\theta)}-\left(\det(A)\right)^{\frac{N}{2}}\frac{\sin((N-1)\theta)}{\sin(\theta)}\right)x_0 
	+a_{12}\left(\det(A)\right)^{\frac{N-1}{2}}\frac{\sin(N\theta)}{\sin(\theta)}y_0. \nonumber
	\end{align}}
If $0<\det(A)\leq1$, then $\left|\mbf E(x_N)\right|\leq K(\theta)$ which leads to
\begin{equation} \label{ExN}
\lim_{N\to\infty}\frac{\mbf E(x_N)}{N}=0.
\end{equation}
From \eqref{xN} and the fact $\hat{\alpha}_{-1}=0$, we get 
\begin{align}
\mbf{Var}(x_N)=&\alpha^2h\sum_{n=0}^{N-1}\left[b_1\hat{\alpha}_{N-1-n}+(\ab)\hat{\alpha}_{N-2-n}\right]^2 \nonumber\\
=&\alpha^2h\left[\left(b_1^2+(\ab)^2\right)\sum_{n=0}^{N-2}\hat{\alpha}^2_n+b_1^2\hat{\alpha}_{N-1}+2b_1(\ab)\sum_{n=1}^{N-1}\hat{\alpha}_{n}\hat{\alpha}_{n-1}\right].\nonumber\\ \label{VarxN}
\end{align}
Further, we have
\begin{gather}
\sum_{n=0}^{N-2}\hat{\alpha}^2_n=\sum_{n=0}^{N-2}\frac{(\det(A))^n\sin^2((n+1)\theta)}{\sin^2(\theta)}, \label{5.1} 
\end{gather}
\begin{align}
2\sum_{n=1}^{N-1}\hat{\alpha}_{n}\hat{\alpha}_{n-1}
=\frac{1}{\sin^2(\theta)}\sum_{n=1}^{N-1}\left(\det(A)\right)^{\frac{2n-1}{2}}\left(\cos(\theta)-\cos((2n+1)\theta)\right).\label{5.2}
\end{align}
As is analogous to the treatment of $\{A_N\}_{N\geq1}$, we  deal with  symplectic methods ($\det(A)=1$) and non-symplectic ones ($0<\det(A)<1$), respectively.
\subsection{LDP of $\{B_N\}_{N\geq1}$ for symplectic methods}
In this part, we study the LDP for $\{B_N\}_{N\geq1}$ of symplectic methods, so we assume that $\mbf{(A2)}$ holds.
Based on  $\det(A)=1$, \eqref{5.1} and \eqref{5.2}, we have
\begin{align}
\sum_{n=0}^{N-2}\hat{\alpha}^2_n=&\frac{1}{\sin^2(\theta)}\sum_{n=1}^{N-1}\sin^2(n\theta)
=\frac{1}{\sin^2(\theta)}\left(\frac{N-1}{2}-\frac{\sin((2N-1)\theta)-\sin(\theta)}{4\sin(\theta)}\right), \label{5.3}
\end{align}
and
\begin{align}
2\sum_{n=1}^{N-1}\hat{\alpha}_n\hat{\alpha}_{n-1}
=\frac{1}{\sin^2(\theta)}\left[(N-1)\cos(\theta)-\frac{\sin(2N\theta)-\sin(2\theta)}{2\sin(\theta)}\right]. \label{5.4}
\end{align}
Substituting \eqref{5.3} and \eqref{5.4} into \eqref{VarxN} yields 
\begin{align}
\mbf{Var}(x_N)=&\alpha^2h\left[\frac{b_1^2+(\ab)^2+2b_1(\ab)\cos(\theta)}{2\sin^2(\theta)}(N-1)\right. \nonumber \\
\phantom{=}&-\frac{\left[b_1^2+(\ab)^2\right]\left[\sin((2N-1)\theta)-\sin(\theta)\right]}{4\sin^3(\theta)}+\frac{b_1^2\sin^2(N\theta)}{\sin^2(\theta)} \nonumber \\
\phantom{=}&\left.-\frac{b_1(\ab)(\sin(2N\theta)-\sin(\theta))}{2\sin^3(\theta)}\right]. \label{VARxN}
\end{align}
Using \eqref{ExN}, \eqref{VARxN} and \eqref{theta} with $\det(A)=1$, we have
\begin{align}
\Lambda^h(\lambda)=&\lim_{N\to\infty}\frac{1}{N}\log\mbf Ee^{\lambda NB_N} \nonumber \\
=&\frac{\alpha^2\lambda^2\left[(b_1+\ab)^2-b_1(\ab)(2-\tr(A))\right]}{\left(4-(\tr(A))^2\right)h}. \label{LambdaB}
\end{align}
Before proving that $\Lambda^h$ is essentially smooth, we give the following lemma (See its proof in the Appendix).
\begin{lem}\label{Sec5lem1}
	Under assumptions $\mbf{(A1)}$ and $\mbf{(A2)}$, it holds that
	$(b_1+\ab)^2-b_1(\ab)(2-\rm{tr}(A))>0.$	
\end{lem}
Lemma \ref{Sec5lem1} shows that
$\Lambda^h(\cdot)$ is  essentially smooth and lower semicontinuous. Then, using Theorem \ref{G-E}, we conclude that $\{B_N\}_{N\geq1}$ satisfies an LDP with the good rate function
\begin{align} 
J^h(y)=\frac{h\left[4-(\tr(A))^2\right]y^2}{4\alpha^2\left[(b_1+\ab)^2-b_1(\ab)(2-\tr(A))\right]}.\label{Jh(y)}
\end{align}

By Definition \ref{Def1.1}, the modified rate function is
\begin{equation}
J_{mod}^h(y)=\frac{\left(4-(\tr(A))^2\right)y^2}{4\alpha^2\left[(b_1+\ab)^2-b_1(\ab)(2-\tr(A))\right]}.
\end{equation}

In what follows, we study the asymptotical convergence of $J_{mod}^h(\cdot)$ as step-size $h$ tends to $0$ based on mean-square convergence condition. To this end, let condition $\mbf{(B)}$ hold. Then 
it follows from Lemma \ref{lem4.3} that
$(2-\tr(A))\sim h^2$, $(b_1+\ab)\sim h.$
In addition, $\mbf{(B)}$ implies that $b_1(\ab)\rightarrow 0$ as $h\rightarrow 0$. In this way, we have
\begin{align}
\lim_{h\to 0}J_{mod}^h(y)
=\frac{2+\lim_{h\to 0}\tr(A)}{4\alpha^2\left[\lim_{h\to 0}\frac{(b_1+\ab)^2}{2-\tr(A)}-\lim_{h\to 0}b_1(\ab)\right]}y^2 
=\frac{y^2}{\alpha^2}. \nonumber
\end{align}

According to the above results, we write them into the following theorem.
\begin{theo} \label{BNLDP}
	For the numerical method \eqref{Mthd} approximating the stochastic oscillator \eqref{Eq1}, if the assumptions $\mbf{(A1)}$ and $\mbf{(A2)}$ hold, then we have 
	
	$\mrm{(1)}$ The method \eqref{Mthd} is symplectic;
	
	$\mrm{(2)}$ The discrete mean velocity $\{B_N\}_{N\geq1}$ of  method \eqref{Mthd} satisfies an LDP with the good rate function
	\begin{equation*} 
	J^h(y)=\frac{h\left[4-(\rm{tr}(A))^2\right]y^2}{4\alpha^2\left[(b_1+\ab)^2-b_1(\ab)(2-\rm{tr}(A))\right]};
	\end{equation*}
	
	$\mrm{(3)}$ Moreover, if assumption $\mbf{(B)}$ holds, then method \eqref{Mthd} asymptotically preserves the LDP of $\{B_T\}_{T>0}$, i.e., the modified rate function $J^h_{mod}(y)=J^h(y)/h$ satisfies:
	\begin{equation*}
	\lim_{h\to 0}J^h_{mod}(y)=J(y),\qquad\forall\quad y\in\mbb R,
	\end{equation*}
	where $J(\cdot)$ is the rate function of the LDP for $\{B_T\}_{T>0}$.
\end{theo}

\subsection{LDP of $\{B_N\}_{N\geq1}$ for non-symplectic methods}
In this part, we consider the discrete mean velocity $\{B_N\}_{N\geq1}$ of general non-symplectic methods. We study whether the LDP holds for $\{B_N\}_{N\geq1}$. Let conditions $\mbf{(A1)}$ and $\mbf{(A3)}$ hold. Then, \eqref{5.1} and \eqref{5.2} satisfy, respectively,
\begin{gather*}
\sum_{n=0}^{N-2}\hat{\alpha}^2_n\leq K(\theta)\sum_{n=0}^{N-2}\left(\det(A)\right)^n\leq K(\theta),\\
\left|2\sum_{n=1}^{N-1}\hat{\alpha}_n\hat{\alpha}_{n-1}\right|\leq K(\theta)\sum_{n=1}^{N-1}\left(\det(A)\right)^{\frac{2n-1}{2}}\leq K(\theta).
\end{gather*}
Additionally, it holds that $\left|\hat{\alpha}_{N-1}\right|=\left|\frac{\left(\det(A)\right)^{N-1}\sin^2(N\theta)}{\sin^2(\theta)}\right|\leq K(\theta)$. Thus, \eqref{VarxN} satisfies 
\begin{equation} \label{521}
\left|\mbf{Var}(x_N)\right|\leq \alpha^2hK(\theta).
\end{equation}
It follows from \eqref{ExN} and \eqref{521} that
the logarithmic moment generating function is
\begin{equation}
\widetilde\Lambda^h(\lambda)=\lim_{N\to\infty}\frac{1}{N}\log\mbf Ee^{\lambda NB_N}=\lim_{N\to\infty}\left[\frac{\lambda}{N}\mbf E(x_N)+\frac{\lambda^2}{2h^2}\mbf{Var}(x_N)\right]=0.
\end{equation}

We note that $\widetilde\Lambda^h(\cdot)$ is not essentially smooth, for which Theorem \ref{G-E} is not valid.  
In our case, we can directly prove that the LDP holds for $\{B_N\}_{N\geq1}$ of non-symplectic methods by the definition of LDP. We claim that $\{B_N\}_{N\geq1}$ of non-symplectic methods satisfy the LDP with the good rate function:
\begin{equation}\label{66k2}
\tilde{J}^h(y)=\left\{
\begin{split}
0, \qquad &y=0,\\
+\infty,\qquad &y\neq0.
\end{split}
\right.
\end{equation}
We divide the proof of this claim into three steps.

\emph{Step 1: We show the limit behavior of $P\left(B_N\geq x_0\right)$ and $P\left(B_N\leq x_0\right)$ for non-symplectic methods.} \\
We need to use the following fact:
if $X\sim\mcal N(\mu,\sigma^2)$, then it follows from \cite[Lemma 22.2]{Achim} that, for any $x>\mu$,
\begin{equation} \label{Nmu}
\mbf P\left(X\geq x
\right)=\mbf P\left(\frac{X-\mu}{\sigma}\geq \frac{x-\mu}{\sigma}
\right)\leq\frac{1}{\sqrt{2\pi}}\frac{\sigma}{x-\mu}e^{-\frac{(x-\mu)^2}{2\sigma^2}}.
\end{equation} 
In addition, for any $x<\mu$,
\begin{equation} \label{Nmu2}
\mbf P\left(X\leq x
\right)=\mbf P\left(\frac{X-\mu}{\sigma}\leq \frac{x-\mu}{\sigma}
\right)=\mbf P\left(\frac{X-\mu}{\sigma}\geq -\frac{x-\mu}{\sigma}
\right)\leq\frac{1}{\sqrt{2\pi}}\frac{\sigma}{\mu-x}e^{-\frac{(x-\mu)^2}{2\sigma^2}}.
\end{equation}

Since $B_N=\frac{x_N}{Nh}$, we have $B_N\sim\mcal N\left(\frac{\mbf E(x_N)}{Nh},\frac{\mbf{Var}(x_N)}{N^2h^2}\right)$ with $\left|\mbf E(x_N)\right|\leq K(\theta)$ and $\left| \mbf{Var}(x_N)\right|\leq K(\theta)$.
Noting that $\underset{N\to\infty}{\lim}\mbf E(B_N)=0$, one has that
for the given $x_0>0$, there exists some $N_0$ such that  $\mbf E(B_N)<x_0$ for every $N>N_0$. Accordingly, it follows from \eqref{Nmu} that
\begin{equation*}
\mbf P\left(B_N\geq x_0\right)\leq \frac{1}{\sqrt{2\pi}}\frac{\sqrt{\mbf{Var}(x_N)}}{Nhx_0-\mbf E(x_N)}\exp\left\{-\frac{\left(Nhx_0-\mbf E(x_N)\right)^2}{2\mbf{Var}(x_N)}\right\},\qquad \forall \quad N>N_0.
\end{equation*}
In this way, for every $x_0>0$,
\begin{equation} \label{523}
\lim_{N\to\infty}\frac{1}{N}\log\left[\mbf P\left(B_N\geq x_0\right)\right]=-\infty.
\end{equation}
Analogously, using \eqref{Nmu2}, one has that for the given $x_0<0$, 
\begin{equation} \label{66k1}
\lim_{N\to\infty}\frac{1}{N}\log\left[\mbf P\left(B_N\leq x_0\right)\right]=-\infty.
\end{equation}

\emph{Step 2: We prove the upper bound LDP (LDP2): For every closed  $C\subset\mbb R$,
	\begin{align}\label{66k3}
	\underset{N\to\infty}{\limsup}\frac{1}{N}\log\mbf P(B_N\in C)\leq-\inf \tilde J^h(C).
	\end{align}
}

If $0\in C$, then it follows from \eqref{66k2} $\inf\tilde J^h(C)=0$. Since $\mbf P(B_N\in C)\leq1$, \eqref{66k3} naturally holds.

If $0\notin C$. Define $x_+=\inf\left(C\bigcap(0,+\infty)\right)$ and $x_-=\sup\left(C\bigcap(-\infty,0)\right)$. Then, $\mbf P(B_N\in C)\leq \mbf P(B_N\geq x_+)+\mbf P(B_N\leq x_-)$. In order to prove \eqref{66k3}, we need to use the following lemma (see \cite[Lemma 23.9]{Achim}).
\begin{lem}\label{lem5.1}
	Let $N\in\mathbb{N}$ and let $a_\epsilon^i$, $i=1,\ldots,N$, $\epsilon>0$, be nonnegative numbers. Then
	$	\limsup_{\epsilon\to 0}\epsilon\log\sum_{i=1}^{N}a_\epsilon^i=\underset{i=1,\ldots,N}{\max}\limsup_{\epsilon\to 0}\epsilon\log(a_\epsilon^i).$	
\end{lem}
Using \eqref{523}， \eqref{66k1} and Lemma \ref{lem5.1} yields
\begin{align*}
\limsup_{N\to\infty}\frac{1}{N}\log\mbf P(B_N\in C)\leq \max\left\{\limsup_{N\to\infty}\frac{1}{N}\log\mbf P(B_N\geq x_+),~\limsup_{N\to\infty}\frac{1}{N}\log\mbf P(B_N\leq x_-)\right\}=-\infty.
\end{align*}
Noting that $0\notin C$, one obtains $\inf\tilde J^h(C)=+\infty$. Thus, \eqref{66k3} also holds for this case.

\emph{Step 3: We prove the lower bound LDP (LDP1): For every open $U\subset\mbb R$,
	\begin{align}\label{66k4}
	\underset{N\to\infty}{\liminf}\frac{1}{N}\log\mbf P(B_N\in U)\geq-\inf \tilde J^h(U).
	\end{align}
}

If $0\notin U$, then $\inf\tilde J^h(U)=+\infty$. Since $\mbf P(B_N\in C)\geq 0$, \eqref{66k4} naturally holds.

If $0\in U$, then there exists some $\delta>0$ such that $(-\delta,\delta)\subset U$. Accordingly, 
\begin{align}\label{66k5}
\liminf_{N\to\infty}\frac{1}{N}\log\mbf P(B_N\in U)\geq \liminf_{N\to\infty}\frac{1}{N}\log\mbf P(|B_N|<\delta).
\end{align}
It follows from \eqref{523} that for arbitrary given $M\in(-\infty,0)$, there exists some $N_1$ such that for every $N>N_1$,
$\frac{1}{N}\log\left[\mbf P\left(B_N\geq \delta\right)\right]<M$. Thus, 
\begin{align*}
\mbf P\left(B_N\geq \delta\right)\leq e^{NM},\qquad \forall \quad N>N_1,
\end{align*}
which leads to $\underset{N\to\infty}{\lim}\mbf P(B_N\geq\delta)=0$. Similarly, utilizing \eqref{66k1} gives  $\underset{N\to\infty}{\lim}\mbf P(B_N\leq-\delta)=0$.
Hence, $\underset{N\to\infty}{\lim}\mbf P(|B_N|<\delta)=1$, which implies
\begin{align}\label{66k6}
\lim_{N\to\infty}\frac{1}{N}\log\mbf P(|B_N|<\delta)=0.
\end{align}
Combining \eqref{66k5} and \eqref{66k6}, we have
$\liminf_{N\to\infty}\frac{1}{N}\log\mbf P(B_N\in U)\geq0.$
Further, since $0\in U$, $\inf\tilde J^h(U)=0$. Hence, we prove \eqref{66k4}.

Combining the above discussion, we deduce that $\{B_N\}_{N\geq1}$ of non-symplectic methods satisfy the LDP with the good rate function $\tilde{J}^h$ given by \eqref{66k2} and the modified rate function $\tilde{J}_{mod}^h=\tilde{J}^h/h=\tilde{J}^h$. Finally, we get the following theorem.
\begin{theo} \label{BNLDP2}
	For the numerical method \eqref{Mthd} approximating the stochastic oscillator \eqref{Eq1}, if the assumptions $\mbf{(A1)}$ and $\mbf{(A3)}$ hold, then we have 
	
	$\mrm{(1)}$ The method \eqref{Mthd} is non-symplectic;
	
	$\mrm{(2)}$ The discrete mean velocity $\{B_N\}_{N\geq1}$ of  method \eqref{Mthd} satisfies an LDP with the good rate function
	$\tilde{J}^h(y)=
	\begin{cases}
	0, \qquad &y=0,\\
	+\infty,\qquad &y\neq0;
	\end{cases}$
	
	$\mrm{(3)}$ Method \eqref{Mthd} does not asymptotically preserve the LDP of $\{B_T\}_{T>0}$, i.e.,  for  $y\neq 0$,
	$	\lim_{h\to 0}\tilde J^h_{mod}(y)\neq J(y),$
	where  $\tilde J^h_{mod}(y)=\tilde J^h(y)/h$, and $J(y)=\frac{y^2}{\alpha^2}$ is the rate function of LDP for $\{B_T\}_{T>0}$.
\end{theo}

\vspace{5mm}
\section{Concrete numerical methods} \label{Sec6}
In this section, we show and compare the LDPs of some concrete numerical methods to verify the theoretical results obtained in previous sections. For symplectic methods, we consider symplectic $\beta$-method, Exponential method, INT method and OPT method. For non-symplectic ones, we examine $\theta$-method, PC (PEM-MR) method and PC (EM-BEM) method. All of the methods can be found in \cite{Tocino15}, except  symplectic $\beta$-method (see e.g., (2.7) in \cite{MilsteinBook}). Furthermore, we construct some symplectic methods which preserve the LDP for $\{A_T\}_{T>0}$ or $\{B_T\}_{T>0}$ exactly.
\subsection{Symplectic methods} \label{Sec6.1}
$\\$

$\bullet$ Symplectic $\beta$-method ($\beta\in[0,1]$):
\begin{equation*} 
A^{\beta}=\frac{1}{1+\beta(1-\beta)h^2}\left(\begin{array}{cc}
1-(1-\beta)^2h^2&h\\
-h&1-\beta^2h^2
\end{array}\right),~
b^{\beta}=\frac{1}{1+\beta(1-\beta)h^2}\left(\begin{array}{c}
(1-\beta)h\\1
\end{array}\right).
\end{equation*}
The straightforward calculation leads to
\begin{gather} 
\det(A^{\beta})=1,\qquad\tr(A^{\beta})=\frac{2-(2\beta^2-2\beta+1)h^2}{1+\beta(1-\beta)h^2}, \label{5beta1}
\\
\ab=\frac{\beta h}{1+\beta(1-\beta)h^2},\qquad b_1+\ab=\frac{h}{1+\beta(1-\beta)h^2}.\label{5beta2}
\end{gather}

It can be verified that condition $\mbf{(B)}$ holds, and if $h\in(0,2)$, then for every $\beta\in[0,1]$, conditions $\mbf{(A1)}$ and $\mbf{(A2)}$ hold. 
Substituting \eqref{5beta1} and \eqref{5beta2} into \eqref{Ih1}, we have
$I^h(y)=\frac{hy^2}{3\alpha^2}\left[\frac{3}{2}-\frac{3}{6-(2\beta-1)^2h^2}\right]$,
which is the good rate function of LDP for $\{A_N\}_{N\geq1}$ of symplectic $\beta$-method by Thoerem \ref{MLDP1}.
Furthermore, we get the modified rate function
$I^h_{mod}(y)=I^h(y)/h=\frac{y^2}{3\alpha^2}\left[\frac{3}{2}-\frac{3}{6-(2\beta-1)^2h^2}\right]$.

Further, we have that $\lim_{h\to 0}I^h_{mod}(y)=I(y)=\frac{y^2}{3\alpha^2}$, for every $y\in\mbb R$, which is consistent with the third conclusion of Theorem \ref{MLDP1}. Moreover, for every $h>0,$ the modified rate function of the mean position for the midpoint method with $\beta=\frac{1}{2}$ is same as that for the exact solution. These indicate that midpoint method exactly preserves the LDP for $\{A_T\}_{T>0}$. In case of $\beta\neq\frac{1}{2}$, $I^h_{mod}(y)<I(y)$ provided $y\neq0$.   That is, as the time $T$ and $t_N$ tend to infinity simultaneously, the exponential decay speed of $\mbf P(A_N\in[a,a+\ud a])$ is slower than that of $\mbf P\left(A_T\in[a,a+\ud a]\right)$ provided $a\neq0$. 

On the other hand, if $h\in(0,2)$ and $\beta\in(0,1)$, conditions $\mbf{(A1)}$ and $\mbf{(A2)}$   hold. By Theorem \ref{BNLDP},  $\{B_N\}_{N\geq1}$ of symplectic $\beta$-method satisfies an LDP with the good rate function $J^h(y)=\frac{h\left[4-(2\beta-1)^2h^2\right]\left[1+\beta(1-\beta)h^2\right]y^2}{4\alpha^2}$. This means that the modified rate function $J^h_{mod}(\cdot)$ satisfies $\lim_{h\to 0}J^h_{mod}(y)=\frac{y^2}{\alpha^2}=J(y)$, which verifies the third conclusion of Theorem \ref{BNLDP}.
\\

$\bullet$Exponential method (EX):
$ A^{EX}=\left(\begin{array}{cc}
\cos(h)&\sin(h)\\
-\sin(h)&\cos(h)
\end{array}\right), \quad
b^{EX}=\left(\begin{array}{c}
0\\1
\end{array}\right).$\\
For this method, it holds that
\begin{gather*}
\det(A^{EX})=1,~\tr(A^{EX})=2\cos(h),~\ab=\sin(h),~ b_1+\ab=\sin(h).
\end{gather*}
If $h\in(0,\pi)$, then the conditions $\mbf{(A1)}$ and $\mbf{(A2)}$   hold. Then, we obtain that $\{A_N\}_{N\geq1}$ satisfies an LDP with the modified rate function
$I^h_{mod}(y)=\frac{2y^2}{\alpha^2}\frac{1-\cos(h)}{h^2(2+\cos(h))}.$
Hence, we have $\lim_{h\to 0}I^h_{mod}(y)=\frac{y^2}{3\alpha^2}=I(y)$. One can show that $I^h_{mod}(y)>I(y)$ provided that $h\in(0,\pi/6)$ and $y\neq0$.


According to the discussions above, if $h\in(0,\pi/6)$, then the mean position $\{A_N\}_{N\geq1}$ of exponential method satisfies an LDP, which asymptotically preserves the LDP for $\{A_T\}_{T>0}$. In addition, as the time $T$ and $t_N$ tend to infinity simultaneously, the exponential decay speed of $\mbf P(A_N\in[a,a+\ud a])$ is faster than that of $\mbf P\left(A_T\in[a,a+\ud a]\right)$ provided that $a\neq0$.

Analogously, we have that conditions $\mbf{(A1)}$ and $\mbf{(A2)}$ hold for $h\in(0,\pi)$. Hence, for $h\in(0,\pi)$, $\{B_N\}_{N\geq1}$ of  exponential method satisfies an LDP with the modified rate function $J_{mod}^h(y)=\frac{y^2}{\alpha^2}=J(y)$. In this way, exponential method exactly preserves the LDP for $\{B_T\}_{T>0}$. 
\\

$\bullet$ Integral method (INT):
$
A^{INT}=\left(\begin{array}{cc}
\cos(h)&\sin(h)\\
-\sin(h)&\cos(h)
\end{array}\right),~
b^{INT}=\left(\begin{array}{c}
\sin(h)\\\cos(h)
\end{array}\right).
$\\
For this method, $\det(A^{INT})=1$, $\tr(A^{INT})=2\cos(h)$, $\ab=0$ and $b_1+\ab=\sin(h)$. It is shown that its modified rate functions of $\{A_N\}_{N\geq1}$ and $\{B_N\}_{N\geq1}$ are 
$I^h_{mod}(y)=\frac{2y^2}{\alpha^2}\frac{1-\cos(h)}{h^2(2+\cos(h))}$, and  $J_{mod}^h(y)=\frac{y^2}{\alpha^2}=J(y)$, respectively.
This case is exactly the same as that of exponential method. 
\\

$\bullet$ Optimal method (OPT):
$
A^{OPT}=\left(\begin{array}{cc}
\cos(h)&\sin(h)\\
-\sin(h)&\cos(h)
\end{array}\right), ~
b^{OPT}=\frac{1}{h}\left(\begin{array}{c}
2\sin^2(\frac{h}{2})\\\sin(h)
\end{array}\right).
$\\
Based on the above two formulas, one has
\begin{gather*}
\det(A^{OPT})=1,\qquad\tr(A^{OPT})=2\cos(h),\qquad\ab=b_1=\frac{1-\cos(h)}{h}.
\end{gather*}
If $h\in(0,\pi)$, then assumptions $\mbf{(A1)}$ and $\mbf{(A2)}$  hold such that $\{A_N\}_{N\geq1}$ of optimal method satisfies an LDP with the modified rate function $I^h_{mod}(y)=\frac{y^2}{3\alpha^2}=I(y)$. Thus, we conclude that the LDP for mean position $\{A_N\}_{N\geq1}$ of  optimal method exactly preserves the LDP for $\{A_T\}_{T>0}$.

The assumptions $\mbf{(A1)}$ and $\mbf{(A2)}$  hold provided that $h\in(0,\pi)$. Thus, for $h\in(0,\pi)$, $\{B_N\}_{N\geq1}$ of optimal method satisfies an LDP with the modified rate function $J_{mod}^h(y)=\frac{h^2y^2}{2(1-\cos(h))\alpha^2}$. Further, we have that $\lim_{h\to 0}J_{mod}^h(y)=\frac{y^2}{\alpha^2}=J(y)$ and $J_{mod}^h(y)>J(y)$. Hence, optimal method asymptotically preserves the LDP for $\{B_T\}_{T>0}$. When the time $T$ and $t_N$ tend to infinity simultaneously, the exponential decay speed of $\mbf P(B_N\in[a,a+\ud a])$ is faster than that of $\mbf P\left(B_T\in[a,a+\ud a]\right)$ provided $a\neq0$.
\subsection{Non-symplectic methods}\label{Sec6.2}
$\\$

$\bullet$ stochastic $\theta$-method ($\theta\in[0,1/2)\cup(1/2,1]$):
\begin{equation*} 
A^{\theta}=\frac{1}{1+\theta^2h^2}\left(\begin{array}{cc}
1-(1-\theta)\theta h^2&h\\
-h&1-(1-\theta)\theta h^2
\end{array}\right), \qquad
b^{\theta}=\frac{1}{1+\theta^2h^2}\left(\begin{array}{c}
\theta h\\1
\end{array}\right).
\end{equation*}
For this method, we have
\begin{gather*}
\det(A^{\theta})=\frac{1+(1-\theta)^2h^2}{1+\theta^2h^2},~1-\tr(A^{\theta})+\det(A^{\theta})=\frac{h^2}{1+\theta^2h^2},~ b_1+\ab=\frac{h}{1+\theta^2h^2}.
\end{gather*}
Notice that $0<\det(A^{\theta})<1$ is equivalent to $\theta\in(1/2,1]$.
One can show that, with $\theta\in(1/2,1]$, $\mbf{(A1)}$, $\mbf{(A3)}$ and $\mbf{(A4)}$
hold for every $h>0$. Hence, for every $\theta\in(1/2,1]$ and $h>0$, the mean position $\{A_N\}_{N\geq1}$ satisfies an LDP with the modified rate function $\widetilde{I}^h_{mod}(y)=\frac{y^2}{2\alpha^2}$, which verifies the third conclusion of Theorem \ref{MLDP2}.
\\

$\bullet$  PC (PEM-MR):
$
A^{1}=\left(\begin{array}{cc}
1-h^2/2&h(1-h^2/2)\\
-h&1-h^2/2
\end{array}\right), \qquad
b^{1}=\left(\begin{array}{c}
h/2\\1
\end{array}\right).$\\
One has that $1-\tr(A^1)+\det(A^1)=h^2-\frac{h^4}{4}$ and $b_1+\ab=h-\frac{h^3}{4}$. We obtain that $\mbf{(A1)}$, $\mbf{(A3)}$ and $\mbf{(A4)}$ hold, provided $h\in(0,\sqrt{2})$. Thus, by Theorem \ref{MLDP2}, $\{A_N\}_{N\geq1}$ of this method satisfies an LDP with the modified rate function
$\widetilde{I}^h_{mod}(y)=\frac{y^2}{2\alpha^2}$.
\\

$\bullet$  PC (EM-BEM):
$A^{2}=\left(\begin{array}{cc}
1-h^2&h\\
-h&1-h^2
\end{array}\right),~
b^{2}=\left(\begin{array}{c}
h\\1
\end{array}\right)$,
which means that $1-\tr(A^2)+\det(A^2)=h^2+h^4$ and $b_1+\ab=h+h^3$. In this case, $\mbf{(A1)}$, $\mbf{(A3)}$ and $\mbf{(A4)}$ hold, provided $h\in(0,1)$. Thus, by Theorem \ref{MLDP2}, $\{A_N\}_{N\geq1}$ of this method satisfies an LDP with the modified rate function
$\widetilde{I}^h_{mod}(y)=\frac{y^2}{2\alpha^2}$.

We observe that all methods shown in Sections \ref{Sec6.1} and  \ref{Sec6.2} satisfy the condition  $\mbf{(B)}$. When the step-size $h$ is sufficiently small, the symplectic methods in Section \ref{Sec6.1} satisfy the conditions $\mbf{(A1)}$ and $\mbf{(A2)}$, and the non-symplectic methods in Section \ref{Sec6.2} satisfy the conditions $\mbf{(A1)}$, $\mbf{(A3)}$ and $\mbf{(A4)}$. 
By studying these methods, we verify the theoretical results in Theorems \ref{MLDP1}, \ref{MLDP2} and \ref{BNLDP}. It is shown that symplectic methods are superior to non-symplectic methods  in terms of preservation of the LDP for both $\{A_T\}_{T>0}$ and $\{B_T\}_{T>0}$.
\subsection{Construction for methods exactly preserving the LDP for $\{A_T\}_{T>0}$ or $\{B_T\}_{T>0}$} \label{Sec6.3}
In this part, we  construct several symplectic methods exactly preserving the LDP for $\{A_T\}_{T>0}$  (resp. $\{B_T\}_{T>0}$) based on Theorem \ref{MLDP1} (resp. Theorem \ref{BNLDP}). 

$\bullet$ Methods exactly preserving the LDP for $\{A_T\}_{T>0}$:

Motivated by assumption  $\mbf{(B)}$, we consider the method \eqref{Mthd} with
\begin{equation} \label{5method1}
A=\left(\begin{array}{cc}
1+c_{11}h^2&h+c_{12}h^2\\
-h+c_{21}h^2&1+c_{22}h^2
\end{array}\right), \qquad
b=\left(\begin{array}{c}
D_1h\\1+D_2h
\end{array}\right)
\end{equation}
with real constants $c_{ij}$ and $D_{i}$, $i,j=1,2$, independent of $h$. In order to make the condition $\det(A)=1$ hold, we have
$$(1+c_{11}h^2)(1+c_{22}h^2)=1+(h+c_{12}h^2)(-h+c_{21}h^2),\quad\forall\quad h>0.$$ 
Comparing the coefficients and we obtain
\begin{gather*}
c_{11}+c_{22}=-1,\quad c_{11}c_{22}=c_{12}c_{21},\quad c_{12}=c_{21}.
\end{gather*}
Let $c_{12}=c_{21}=\sigma$, then $c_{11}$ and $c_{22}$ are the roots of equation $x^2+x+\sigma^2=0$. To assure that $c_{11}$ and $c_{22}$ are real numbers, we assume $\sigma\in[-1/2,1/2]$.
Solving the equation $x^2+x+\sigma^2=0$ yields
$
c_{11}=\frac{-1-\sqrt{1-4\sigma^2}}{2}$, $c_{22}=\frac{-1+\sqrt{1-4\sigma^2}}{2}$ or $c_{11}=\frac{-1+\sqrt{1-4\sigma^2}}{2}$, $c_{22}=\frac{-1-\sqrt{1-4\sigma^2}}{2},$
where the case $c_{11}=c_{22}=-1/2$, $\sigma=\pm1/2$ is included in the above two cases.
In order to acquire the methods exactly preserving the LDP for $\{A_T\}_{T>0}$, a necessary condition is that the modified rate function \eqref{mod1} satisfies
\begin{equation} \label{5modI}
I_{mod}^h(y)=\frac{(2+\mrm{tr}(A))(2-\mrm{tr}(A))^2y^2}{2\alpha^2h^2\left[(b_1+\ab)^2(4+\mrm{tr}(A))-2b_1(\ab)(2-\mrm{tr}(A))\right]}=\frac{y^2}{3\alpha^2}.
\end{equation} 
According to \eqref{5method1}, it is known that
\begin{gather*} 
\tr(A)=2-h^2,\quad \ab=h\left[(1-D_1)+(D_2+\sigma )h+(D_2\sigma-D_1c_{22})h^2\right].
\end{gather*}
Substituting the above equation into \eqref{5modI}, we have
\begin{align} \label{key2}
6-\frac{3h^2}{2}=&\left[1+(D_2+\sigma)h+(D_2\sigma-D_1c_{22})h^2\right]^2(6-h^2) \nonumber \\
\phantom{=}&-2D_1h^2\left[1-D_1+(D_2+\sigma)h+(D_2\sigma-D_1c_{22})h^2\right].
\end{align}
By comparing the coefficients of $h^6$ and $h^4$ in \eqref{key2} and some direct computation, we finally obtain  
\begin{equation*} \label{key7}
D_1=\frac{1}{2},\quad\sigma=0,\pm\frac{1}{2},\quad c_{22}=\frac{-1+\sqrt{1-4\sigma^2}}{2},\quad c_{11}=\frac{-1-\sqrt{1-4\sigma^2}}{2},\quad D_2=-\sigma.
\end{equation*}

Finally, we acquire the following three methods exactly preserving the LDP for $\{A_T\}_{T>0}$ and their coefficients are separately 
\begin{equation} \label{M1}
A^{[1]}=\left(\begin{array}{cc}
1-h^2&h\\
-h&1
\end{array}\right), \qquad
b^{[1]}=\left(\begin{array}{c}
h/2\\1
\end{array}\right);
\end{equation}
\begin{equation} \label{M2}
A^{[2]}=\left(\begin{array}{cc}
1-h^2/2&h+h^2/2\\
-h+h^2/2&1-h^2/2
\end{array}\right), \qquad
b^{[2]}=\left(\begin{array}{c}
h/2\\1-h/2
\end{array}\right);
\end{equation}
\begin{equation} \label{M3}
A^{[3]}=\left(\begin{array}{cc}
1-h^2/2&h-h^2/2\\
-h-h^2/2&1-h^2/2
\end{array}\right), \qquad
b^{[3]}=\left(\begin{array}{c}
h/2\\1+h/2
\end{array}\right).
\end{equation}
Moreover, if $h\in(0,2)$, methods based on \eqref{M1}, \eqref{M2} and \eqref{M3} satisfy the assumptions $\mbf{(A1)}$ and $\mbf{(A2)}$,  and  have the same modified rate function $I^h_{mod}(y)=\frac{y^2}{3\alpha^2}=I(y)$.
\\

$\bullet$ Methods exactly preserving the LDP for $\{B_T\}_{T>0}$:

We still consider the method with coefficients satisfying \eqref{5method1}.  By the straightforward computation, we get the following  methods exactly preserving the LDP for $\{B_T\}_{T>0}$, whose coefficients are
\begin{equation*} 
A=\left(\begin{array}{cc}
1-\frac{1+\sqrt{1-4\sigma^2}}{2}h^2&h+\sigma h^2\\
-h+\sigma h^2&1-\frac{1-\sqrt{1-4\sigma^2}}{2}h^2
\end{array}\right), \qquad
b=\left(\begin{array}{c}
h/2\\1-\sigma h
\end{array}\right),
\end{equation*}
with $\sigma=0,\pm\frac{1}{2}$, or
\begin{equation*} 
A=\left(\begin{array}{cc}
1-\frac{1-\sqrt{1-4\sigma^2}}{2}h^2&h+\sigma h^2\\
-h+\sigma h^2&1-\frac{1+\sqrt{1-4\sigma^2}}{2}h^2
\end{array}\right), \qquad
b=\left(\begin{array}{c}
-h/2\\1-\sigma h
\end{array}\right),
\end{equation*}
with $\sigma=0,\pm\frac{1}{2}$.
Finally, besides methods based on \eqref{M1}, \eqref{M2} and \eqref{M3}, we obtain three more methods exactly preserving the LDP for $\{B_T\}_{T>0}$ with coefficients given by
\begin{equation} \label{M4}
A^{[4]}=\left(\begin{array}{cc}
1&h\\
-h&1-h^2
\end{array}\right), \qquad
b^{[4]}=\left(\begin{array}{c}
-h/2\\1
\end{array}\right);
\end{equation}
\begin{equation} \label{M5}
A^{[5]}=\left(\begin{array}{cc}
1-h^2/2&h+h^2/2\\
-h+h^2/2&1-h^2/2
\end{array}\right), \qquad
b^{[5]}=\left(\begin{array}{c}
-h/2\\1-h/2
\end{array}\right);
\end{equation}
\begin{equation} \label{M6}
A^{[6]}=\left(\begin{array}{cc}
1-h^2/2&h-h^2/2\\
-h-h^2/2&1-h^2/2
\end{array}\right), \qquad
b^{[6]}=\left(\begin{array}{c}
-h/2\\1+h/2
\end{array}\right).
\end{equation}
In fact, it is verified that methods based on \eqref{M1}, \eqref{M2}, \eqref{M3}, \eqref{M4}, \eqref{M5} and \eqref{M6} satisfy the assumptions $\mbf{(A1)}$ and $\mbf{(A2)}$  for $h\in(0,2)$ and  have the same modified rate function $J^h_{mod}(y)=\frac{y^2}{\alpha^2}=J(y)$.
\begin{rem}
	Note that three symplectic methods constructed based on \eqref{M4}， \eqref{M5} and \eqref{M6} preserve exactly  the LDP for $\{A_T\}_{T>0}$ and $\{B_T\}_{T>0}$ at the same time.
\end{rem}

\section{Conclusions and future aspects}\label{Sec7}

In this paper, in order to evaluate the ability of the numerical method to preserve the large deviations rate functions associated with the general stochastic Hamiltonian systems, we propose the concept of asymptotical preservation for LDPs. 
It is shown that stochastic symplectic methods applied to the stochastic test equation, that is, the linear stochastic oscillator, asymptotically preserve the LDPs for $\{A_T\}_{T>0}$ and $\{B_T\}_{T>0}$, but non-symplectic ones do not.
This indicates the probabilistic superiority of stochastic symplectic methods.
In fact, there are still many problems of interest which remain  to be solved. We list some possible aspects for future work.
\begin{itemize}
	\item[(1)] Can the stochastic  symplectic methods asymptotically preserve the LDPs for all observables associated with the linear stochastic oscillator?
	\item[(2)] Can stochastic  symplectic methods asymptotically preserve the LDPs for observables associated with the general stochastic Hamiltonian system which is driven by multiplicative noises or in higher dimension?
	\item[(3)] For a stochastic Hamiltonian partial differential equation which possesses the symplectic or multi-symplectic structure, such as stochastic Schr\"odinger equation, 
	does the symplectic or multi-symplectic numerical methods asymptotically preserve the LDP of the original system?
\end{itemize}
These problems are very challenging. Because the large deviations rate functions do not generally  have explicit expression for more complex SDEs and their numerical solutions, it is difficult to  analyze the asymptotical behaviour of rate functions of numerical methods. In addition, the  large deviations estimates on infinite dimensional Banach spaces are more involved. We leave these problems as the open problems, and attempt to study them in our future work.

\section*{Appendix}
\textbf{A. Proof of Lemma \ref{lem1}}.
\begin{proof}
	Using the fact $\sin(n\theta)=\frac{1}{2\bm{i}}\left(e^{\bm{i}n\theta}-e^{-\bm{i}n\theta}\right)$, one immediately has
	\begin{align}
	\sum_{n=1}^{N}\sin(n\theta)a^n
	=\frac{a\sin(\theta)-a^{N+1}\sin((N+1)\theta)+a^{N+2}\sin(N\theta)}{1-2a\cos(\theta)+a^2}.\nonumber
	\end{align}
	
	For $a=1$, utilizing the formula $\sin(\alpha)-\sin(\beta)=2\cos(\frac{\alpha+\beta}{2})\sin(\frac{\alpha-\beta}{2})$ gives
	\begin{align}
	\sum_{n=1}^{N}\sin(n\theta)=&\frac{\sin(\theta)-\sin((N+1)\theta)+\sin(N\theta)}{2(1-\cos(\theta))} 
	=\frac{\cos\left(\frac{\theta}{2}\right)-\cos((N+\frac{1}{2})\theta)}{2\sin\left(\frac{\theta}{2}\right)}, \nonumber
	\end{align}
	which completes the proof.
\end{proof}

\textbf{B. Proof of Lemma \ref{lem3.2}}.
\begin{proof}
	(1) Assume that $b_1^2+(\ab)^2=0$, i.e., $b_1=\ab=0$. Noting that $b_1^2+b_2^2\neq0$, one has $b_2\neq0$, which leads to $a_{12}=0.$ Since $\det(A)=a_{11}a_{22}-a_{12}a_{21}=1$, $a_{11}a_{22}=1$. Hence $a_{11},a_{22}>0$ or $a_{11},a_{22}<0$. It follows from assumptions $\mbf{(A1)}$ and $\mbf{(A2)}$ that $-2<\tr(A)<2$. In this way, $|\tr(A)|=|a_{11}|+|a_{22}|<2$. This is contradictory  to $|a_{11}a_{22}|=1$, since $1=\sqrt{|a_{11}a_{22}|}\leq\frac{1}{2}\left(|a_{11}|+|a_{22}|\right)<1.$ This proves the first conclusion.
	
	(2) Denote $S:=(b_1+\ab)^2(4+\tr(A))-2b_1(\ab)(2-\tr(A))$, $p:=b_1$ and $q:=\ab$. Then $S=(p+q)^2(4+\tr(A))-2pq(2-\tr(A))=\tr(A)\left((p+q)^2+2pq\right)+4(p+q)^2-4pq$. By studying the infimum of $S$ in three kinds of cases: $(p+q)^2+2pq>0$, $(p+q)^2+2pq<0$ and $(p+q)^2+2pq=0$, one can prove that $S>0$.	
	%
	%
\end{proof}

\textbf{C. Proof of Theorem \ref{tho4.1}}.
\begin{proof}
	Denote $Z_t=(X_t,Y_t)$, $J=\left(
	\begin{array}{cc}
	0 & 1\\
	-1& 0
	\end{array}
	\right)$, $K=\left(
	\begin{array}{cc}
	0 \\
	1
	\end{array}
	\right)$. We rewrite \eqref{Eq1} as
	$\ud Z_t=JZ_t\ud t+\alpha K\ud W_t.$
	Let $Z$ be the solution of the above equation at  $t+h$, with the deterministic value $z$ at  $t$. Note that for any $u>v\geq0$, 
	$		Z_u=Z_v+J\int_{v}^{u}Z_r\ud r+\alpha K\int_{v}^{u}\ud W_r.$
	Using the above formula, one can show that 
	\begin{align}\label{sec7k3}
	Z
	=&z+hJz+\alpha K(W_{t+h}-W_t)+J^2\int_{t}^{t+h}\int_{t}^{s}Z_r\ud r\ud s+\alpha JK\int_{t}^{t+h}\int_{t}^{s}\ud W_r\ud s\nonumber\\
	=&A^{EM}z+\alpha b^{EM}(W_{t+h}-W_t)+R,
	\end{align}
	where $R:=J^2\int_{t}^{t+h}\int_{t}^{s}Z_r\ud r\ud s+\alpha JK\int_{t}^{t+h}\int_{t}^{s}\ud W_r\ud s$ with $\|\mbf ER\|_2\leq Ch^2$ and $\mbf E\|R\|^2_2\leq Ch^3$. 
	Further, the one-step approximations based on  the method \eqref{Mthd} is 
	$	\widehat {Z}=Az+\alpha b(W_{t+h}-W_t).$
	In this way, we obtain
	\begin{equation} \label{err1}
	\left\|\mbf E(\widehat{Z}-Z)\right\|_2\leq C\left\|A-A^{EM}\right\|_F\left\|z\right\|_2+\|\mbf ER\|_2\leq Ch^2,
	\end{equation}
	where the second equality uses the equivalence of norms in finite-dimensional normed linear spaces.
	In addition, it holds that
	\begin{align} \label{err2}
	\mbf E\left\|\widehat{Z}-Z\right\|^2_2\leq C\left\|A-A^{EM}\right\|^2_F\left\|z\right\|^2_2+C\alpha^2\left\|b-b^{EM}\right\|^2_2\mbf E(\DW^2)+C\mbf E\|R\|^2_2
	\leq Ch^3.
	\end{align}
	It follows from \eqref{err1}, \eqref{err2} and \cite[Theorem 1.1]{MilsteinBook} that the mean-square order of numerical method
	\eqref{Mthd} is at least $1$.
\end{proof}

\textbf{D. Proof of Lemma \ref{lem4.3}}.
\begin{proof}
	If $\mbf{(B)}$ holds, then $a_{11}=1+\mcal O(h^2)$, $a_{22}=1+\mcal O(h^2)$. Thus, $\mrm{tr}(A)=2+\mcal O(h^2)$, which leads to the assertion $(1)$.
	Further,
	$1-\mrm{tr}(A)+\det(A)
	=(a_{11}-1)(a_{22}-1)-a_{12}a_{21}.$
	Noting that $a_{12}\sim h$ and $a_{21}\sim-h$, one has $\left(1-\mrm{tr}(A)+\det(A)\right)\sim h^2$. 
	Finally, since 
	$\lim_{h\to 0}\frac{a_{12}b_2}{h}=\lim_{h\to 0}\frac{(a_{12}-h)(b_2-1)+h(b_2-1)+a_{12}}{h}=1,$
	it holds that
	$\lim_{h\to 0}\frac{b_1+\ab}{h}=\lim_{h\to 0}\frac{a_{12}b_2}{h}+\lim_{h\to 0}\frac{b_1(1-a_{22})}{h}=1
	$,
	which is nothing but the assertion (3).
\end{proof}

\textbf{E. Proof of Lemma \ref{Sec5lem1}}.
\begin{proof}
	It follows form Lemma \ref{lem3.2}(1) that $b_1^2+(\ab)^2\neq0$. Denote $T=(b_1+\ab)^2-b_1(\ab)(2-\tr(A))$.
	Then $T=b_1^2+(\ab)^2+b_1(\ab)\tr(A)$.	Next we show that $T>0$.
	
	\emph{Case 1: $b_1=0$ or $\ab=0$.} This associated with $b_1^2+(\ab)^2\neq0$ immediately leads to $T>0.$
	
	\emph{Case 2: $b_1\neq0$ and $\ab\neq0$.}
	We note that under assumptions $\mbf{(A1)}$ and  $\mbf{(A2)}$, $-2<\tr(A)<2$. If $\tr(A)=0$, $T>0$ holds naturally. If $\tr(A)\neq0$, then $0<|\tr(A)|<2$. As a result, 
	$\left|b_1(\ab)\tr(A)\right|<2\left|b_1(\ab)\right|\leq b_1^2+(\ab)^2.$	
	Hence,  to sum up, $T>0$.
\end{proof}

\section*{Acknowledgements}
We would like to thank Prof. Xia Chen for lecturing courses on LDP in the summer of 2018.

%

\bibliography{Mybibfile}

\begin{thebibliography}{10}

\bibitem{CohenVil}
A.~Abdulle, D.~Cohen, G.~Vilmart, and K.C. Zygalakis.
\newblock High weak order methods for stochastic differential equations based
  on modified equations.
\newblock {\em SIAM J. Sci. Comput.}, 34(3):A1800--A1823, 2012.

\bibitem{Anton19}
C.~Anton.
\newblock Weak backward error analysis for stochastic {H}amiltonian {S}ystems.
\newblock {\em BIT}, 59(3):613--646, 2019.

\bibitem{CHRK}
C.~Chen and J.~Hong.
\newblock Symplectic {R}unge-{K}utta semidiscretization for stochastic
  {S}chr\"{o}dinger equation.
\newblock {\em SIAM J. Numer. Anal.}, 54(4):2569--2593, 2016.

\bibitem{Huang}
C.~Chen, J.~Hong, and C.~Huang.
\newblock Stochastic modified equations for symplectic methods applied to rough
  {H}amiltonian systems based on the {W}ong--{Z}akai approximation.
\newblock {\em arXiv:1907.02825}, 2019.

\bibitem{CHJ}
C.~Chen, J.~Hong, and L.~Ji.
\newblock Mean-square convergence of a symplectic local discontinuous
  {G}alerkin method applied to stochastic linear {S}chr\"{o}dinger equation.
\newblock {\em IMA J. Numer. Anal.}, 37(2):1041--1065, 2017.

\bibitem{CHZ}
C.~Chen, J.~Hong, and L.~Zhang.
\newblock Preservation of physical properties of stochastic {M}axwell equations
  with additive noise via stochastic multi-symplectic methods.
\newblock {\em J. Comput. Phys.}, 306:500--519, 2016.

\bibitem{ChenX}
X.~Chen.
\newblock {\em Random walk intersections. Large deviations and related topics},
  volume 157 of {\em Mathematical Surveys and Monographs}.
\newblock American Mathematical Society, Providence, RI, 2010.

\bibitem{Cohen}
D.~Cohen.
\newblock On the numerical discretisation of stochastic oscillators.
\newblock {\em Math. Comput. Simulation}, 82(8):1478--1495, 2012.

\bibitem{CHLZ}
J.~Cui, J.~Hong, Z.~Liu, and W.~Zhou.
\newblock Stochastic symplectic and multi-symplectic methods for nonlinear
  {S}chr\"{o}dinger equation with white noise dispersion.
\newblock {\em J. Comput. Phys.}, 342:267--285, 2017.

\bibitem{Faou}
A.~Debussche and E.~Faou.
\newblock Weak backward error analysis for {SDE}s.
\newblock {\em SIAM J. Numer. Anal.}, 50(3):1735--1752, 2012.

\bibitem{Dembo}
A.~Dembo and O.~Zeitouni.
\newblock {\em Large deviations techniques and applications}, volume~38 of {\em
  Stochastic Modelling and Applied Probability}.
\newblock Springer-Verlag, Berlin, 2010.
\newblock Corrected reprint of the second (1998) edition.

\bibitem{Anton2}
J.~Deng, C.~Anton, and Y.S. Wong.
\newblock High-order symplectic schemes for stochastic {H}amiltonian systems.
\newblock {\em Commun. Comput. Phys.}, 16(1):169--200, 2014.

\bibitem{Adaptive}
G.~Ferr\'{e} and H.~Touchette.
\newblock Adaptive sampling of large deviations.
\newblock {\em J. Stat. Phys.}, 172(6):1525--1544, 2018.

\bibitem{HSW}
J.~Hong, L.~Sun, and X.~Wang.
\newblock High order conformal symplectic and ergodic schemes for the
  stochastic {L}angevin equation via generating functions.
\newblock {\em SIAM J. Numer. Anal.}, 55(6):3006--3029, 2017.

\bibitem{HongW}
J.~Hong and X.~Wang.
\newblock {\em Invariant measures for stochastic nonlinear {S}chr\"{o}dinger
  equations. Numerical approximations and symplectic structures}, volume 2251
  of {\em Lecture Notes in Mathematics}.
\newblock Springer, Singapore, 2019.

\bibitem{Achim}
A.~Klenke.
\newblock {\em Probability theory. A comprehensive course}.
\newblock Universitext. Springer-Verlag London, Ltd., London, 2008.

\bibitem{Mxr}
X.~Mao.
\newblock {\em Stochastic differential equations and applications}.
\newblock Horwood Publishing Limited, Chichester, second edition, 2008.

\bibitem{MilsteinBook}
G.N. Milstein and M.V. Tretyakov.
\newblock {\em Stochastic numerics for mathematical physics}.
\newblock Scientific Computation. Springer-Verlag, Berlin, 2004.

\bibitem{OU}
E.S. Palamarchuk.
\newblock Analytic study of an {O}rnstein-{U}hlenbeck process with variable
  coefficients for modeling anomalous diffusions.
\newblock {\em Avtomat. i Telemekh.}, (2):109--121, 2018.

\bibitem{Tocino15}
M.J. Senosiain and A.~Tocino.
\newblock A review on numerical schemes for solving a linear stochastic
  oscillator.
\newblock {\em BIT}, 55(2):515--529, 2015.

\bibitem{Tony}
T.~Shardlow.
\newblock Modified equations for stochastic differential equations.
\newblock {\em BIT}, 46(1):111--125, 2006.

\bibitem{Wang2014}
L.~Wang and J.~Hong.
\newblock Generating functions for stochastic symplectic methods.
\newblock {\em Discrete Contin. Dyn. Syst.}, 34(3):1211--1228, 2014.

\bibitem{WHS}
L.~Wang, J.~Hong, and L.~Sun.
\newblock Modified equations for weakly convergent stochastic symplectic
  schemes via their generating functions.
\newblock {\em BIT}, 56(3):1131--1162, 2016.

\bibitem{Zyg}
K.C. Zygalakis.
\newblock On the existence and the applications of modified equations for
  stochastic differential equations.
\newblock {\em SIAM J. Sci. Comput.}, 33(1):102--130, 2011.

\end{thebibliography}
\bibliographystyle{plain}

\end{document}